\documentclass[12pt]{amsart}
\usepackage[english]{babel}
\usepackage{graphicx}
\usepackage{amsmath,amssymb,hyperref,srcltx}
\newcommand{\diff}[2]{\mbox{{\rm Diff}{${\,}_{#1}({\mathbb C}^{#2},0)$}}}
\newcommand{\diffh}[2]{\mbox{$\widehat{\rm Diff}{{\,}_{#1}({\mathbb C}^{#2},0)}$}}

\newcommand{\cn}[1]{\mbox{(${\mathbb C}^{#1},0$)}}

\newcommand{\pn}[1]{{\mathbb P}^{#1}({\mathbb C})}

\newtheorem{pro}{Proposition}[section]
\newtheorem{teo}{Theorem}[section]
\newtheorem*{main}{Main Theorem}
\newtheorem{cor}{Corollary}[section]

\newtheorem{lem}{Lemma}[section]

\theoremstyle{remark}
\newtheorem{rem}{Remark}[section]

\theoremstyle{definition}
\newtheorem{defi}{Definition}[section]

\begin{document}

\title[The solvable length of groups of local diffeomorphisms]
{The solvable length of groups of local diffeomorphisms}

\author{Javier Rib\'{o}n}
\address{Instituto de Matem\'{a}tica, UFF, Rua M\'{a}rio Santos Braga S/N
Valonguinho, Niter\'{o}i, Rio de Janeiro, Brasil 24020-140}
\thanks{e-mail address: javier@mat.uff.br}
\thanks{MSC class. Primary: 37F75, 20F16; Secondary: 20F14, 32H50}
\maketitle

\bibliographystyle{plain}
\section*{Abstract}
We are interested in the algebraic properties of groups of local biholomorphisms and their 
consequences. A natural question is whether the complexity of solvable groups is bounded 
by the dimension of the ambient space. In this spirit we
show that $2n+1$ is the sharpest upper bound for the derived length of solvable subgroups
of the group $\mathrm{Diff}({\mathbb C}^{n},0)$ of local complex analytic diffeomorphisms
for $n=2,3,4,5$.
\section{Introduction}
It is a well known result of Zassenhaus (1938) that the soluble length (also called derived length)
of a solvable group of $\mathrm{GL}(n,F)$ is bounded by a function of
$n$ that does not depend on the field $F$ or the group \cite{Zassenhaus}.
The sharpest upper bound $\rho (n)$
for the derived length of linear groups in dimension $n$ is called the
Newman function \cite{Newman}.

Let ${\rm Diff} ({\mathbb C}^{n},0)$ be the group of germs of complex analytic diffeomorphisms defined in
a neighborhood of $0 \in {\mathbb C}^{n}$. We are interested in understanding
how the algebraic complexity of solvable subgroups of local diffeomorphisms
is bounded in terms of the dimension $n$.
This is a non-obvious question since the group of local diffeomorphisms is infinitely dimensional.
In line with such a goal we study
the solvability properties of subgroups of $\diff{}{n}$ and more precisely
the sharpest upper
bound for the derived length of solvable subgroups of $\diff{}{n}$, i.e.
the function
\[ \psi(n) = \max \{ \ell (G) : G \ \mathrm{is} \ \mathrm{a} \ \mathrm{solvable} \
\mathrm{subgroup} \ \mathrm{of} \ {\rm Diff} ({\mathbb C}^{n},0) \} \]
where $\ell (G)$ is the derived length of $G$ (cf. Definition \ref{def:length}).
Mitchael Martelo and the author proved $\psi (n) \leq \rho (n) + 2n-1$ in \cite{JR:arxivdl}.
For example the function $\rho (n) + 2n-1$ is equal to
$2$, $7$, $10$, $13$ and $16$ for $n=1$, $2$, $3$, $4$ and $5$ respectively.
It is known that $2$ is a sharp bound for $n=1$
(cf. \cite{Loray5} \cite{Ilya-Yako}[Theorem 6.10 and Remark 6.19]),
in particular every solvable subgroup of $\diff{}{}$ is metabelian.
No other value of the function $\psi$ was known. In this paper we prove:
\begin{main}
\label{teo:main}
Fix $n \in \{2,3,4,5\}$.
Let $G$ be a solvable subgroup of $\diff{}{n}$ (or $\diffh{}{n}$).
Then its derived length $\ell (G)$ is at most $2n+1$.
Moreover there exists a solvable subgroup
$G$ of $\diff{}{n}$ such that $\ell (G)=2n+1$.
\end{main}
In other words we obtain $\psi (n)=2n+1$ for $2 \leq n \leq 5$.
The notation $\diffh{}{n}$ stands for the group of formal diffeomorphisms
(cf. section \ref{sec:fordif}). We just remark now that our techniques
work as well in the analytic or the formal category.
Notice that $\psi (n)=2n+1$ does not hold for $n=1$.

The integrability properties of differential equations are related to
the algebraic properties, and specially
solubility, of holonomy groups of local diffeomorphisms.
More precisely integrability
can be interpreted in terms
of certain (Galois, holonomy) groups of symmetries of the equation.
This connection is explicit in the context of differential Galois theory.
For example Kolchin showed that a linear differential equation admits a liouvillian extension by
integrals and exponentials of integrals if and only if its Galois group is
solvable \cite{Kolchin-PV}.
This phenomenon is not restricted to linear differential equations.
We can find other instances of the principle in
the study of analytic Hamiltonian systems
\cite[Ziglin]{Ziglin1, Ziglin2} \cite[Morales-Ruiz, Ramis, Simo]{Mor-Ram-Simo}, codimension $1$ analytic foliations
\cite[Mattei and Moussu]{MaMo:Aen}
\cite[Paul]{Paul:pre}
and higher codimension analytic foliations \cite[Rebelo and Reis]{RR:arxiv}.

The Main Theorem
can be applied to study the representations of solvable groups in
the group of real analytic diffeomorphisms of a manifold.
In particular it provides a negative criterion for a solvable group
to be embedded in $\diff{}{n}$.
It is part of a program intended to describe algebraic properties
of solvable groups of (complex, real) analytic diffeomorphisms of (complex, real)
analytic manifolds.
We present an application in the real analytic case.
\begin{lem}
\label{lem:dlfp}
Let $G$ be a solvable group of real analytic diffeomorphisms of a
connected manifold $M^{n}$ of dimension $n$. Suppose that $G$
has a global fixed point $p$, i.e. $p$ is a fixed
point of every element of $G$. Then  $\ell (G) \leq \psi (n)$.
\end{lem}
\begin{proof}
We include the proof for the sake of completeness. Consider the sequence of maps
\[ G \stackrel{i}{\hookrightarrow} \mathrm{Diff}^{\omega} (M^{n}, p) \stackrel{j}{\hookrightarrow} \mathrm{Diff} \cn{n} \]
where $\mathrm{Diff}^{\omega} (M^{n}, p)$ is the group of germs of real analytic diffeomorphisms defined in a
neighborhood of $p$ in $M^{n}$. The map $i$ associates to $\phi \in G$ its germ in a
neighborhood of $p$. The map $j$ is obtained by complexification of $M^{n}$, in this way
we can interpret germs of real analytic maps as germs of holomorphic maps.
Both maps $i$ and $j$ are injective. Hence we obtain
$\ell (G)= \ell ( (j \circ i) (G)) \leq \psi (n)$.
\end{proof}
The next result is an immediate consequence of the Main Theorem and Lemma \ref{lem:dlfp}.
\begin{cor}
\label{cor:realan}
Let $G$ be a solvable group of real analytic diffeomorphisms of a connected manifold $M^{n}$
where $2 \leq n \leq 5$. Suppose that $G$ has a global fixed point $p$.
Then  $\ell (G) \leq 2 n +1$.
\end{cor}
Some interesting results about the algebraic properties of solvable groups
of area-preserving real analytic diffeomorphisms of surfaces have been obtained by Franks and Handel.
For instance they showed that for the $2$-sphere such a group
is always virtually-abelian \cite{Franks-Handel:va}.
%

As a consequence of Corollary \ref{cor:realan}
it is interesting to study the existence of common fixed points
or finite orbits for a solvable group of analytic diffeomorphisms (and its derived subgroups).
This point of view, finding a finite orbit of the group to profit from the algebraic
properties of groups of local diffeomorphisms, was used by Ghys to show that any
nilpotent subgroup of real analytic diffeomorphisms of the $2$-sphere is
metabelian \cite{Ghys-identite}.
Another application of groups of local diffeomorphisms to a global problem is the study
of faithful analytic actions of mapping class groups of surfaces on surfaces \cite{Cantat-Cerveau}.
\subsection{Pro-algebraic closure of a group of local diffeomorphisms}
One of the main ingredientes in the paper is the notion of pro-algebraic
group of (formal) local diffeomorphisms. These groups are (possibly)
infinite dimensional groups that are obtained as projective limits of affine
algebraic groups.
The concept of pro-algebraic group is the analogue in the context of local diffeomorphisms
of algebraic matrix groups.
The upside of this concept is the existence of
many similarities
between the theories of pro-algebraic groups and linear algebraic groups.
Pro-algebraic groups have independent interest
besides the study of the derived length, for instance
groups defined by invariance properties are typically pro-algebraic
(Remark \ref{rem:invpro}).
Section \ref{sec:forgroup} is devoted to introduce pro-algebraic groups,
explain their basic properties, characterizing them in certain cases and
to provide methods to construct examples (Proposition \ref{pro:consag}).

The main difficulty to work with subgroups of ${\rm Diff} ({\mathbb C}^{n},0)$
is the infinite dimensional nature of ${\rm Diff} ({\mathbb C}^{n},0)$.
Anyway any subgroup $G$ of ${\rm Diff} ({\mathbb C}^{n},0)$
is contained in a minimal pro-algebraic group $\overline{G}^{(0)}$ containing $G$,
the so called (pro-)algebraic closure of $G$
(\cite{JR:arxivdl}, see also  section \ref{sec:forgroup}).
The pro-algebraic closure is the analogue in the context of local diffeomorphisms
of the Zariski closure of linear groups.
Moreover the derived lengths of $G$ and $\overline{G}^{(0)}$ coincide \cite{JR:arxivdl}.
The pro-algebraic closure of ${\rm Diff} ({\mathbb C}^{n},0)$ is called
group of formal diffeomorphisms.
With this point of view we can apply
techniques of Lie theory and the theory of linear
algebraic groups to $\overline{G}^{(0)}$.
For instance
we can associate a Lie algebra ${\mathfrak g}$
(of formal vector fields) to $\overline{G}^{(0)}$ \cite{JR:arxivdl}.
We replace $G$ with $\overline{G}^{(0)}$ in the remainder of the introduction.
Even if we do not define $\overline{G}^{(0)}$ here, for now the reader
can consider it as a technical device intended to place $G$ in
the context of algebraic group theory. We prefer to introduce some of the properties
of $\overline{G}^{(0)}$ and to give an idea
of how they can be applied to prove the Main Theorem
instead of a more formal presentation.
The details can be found in next sections.
We suppose also that  $\overline{G}^{(0)}$ is a subgroup of
$\mathrm{Diff} ({\mathbb C}^{n},0)$. This is not true in general
since $\overline{G}^{(0)}$ is a group of formal diffeomorphisms but
the differences between this case and the general one are purely technical.

\subsection{First estimation of $\psi$}
We make an outline of the paper in the remainder of the introduction.
We explain the ideas behind our approach, formal definitions and proofs
are delayed for next sections.

The algebraic properties of the Lie algebra ${\mathfrak g}$ of $G$ and $G$
are analogous if $G$ is connected, i.e. if
the image of the map $j^{1}: G \to \mathrm{GL}(n,{\mathbb C})$,
that sends every diffeomorphism in its linear part, is connected in
the Zariski topology (or equivalently in the usual topology
since $j^{1} G$ is a linear algebraic group).
In this way we proved $\ell (G) \leq 2n$ if $G$ is a connected subgroup of
${\rm Diff} ({\mathbb C}^{n},0)$;
moreover we obtain $\ell (G) \leq 2n-1$ if $j^{1} G$ is a group of unipotent matrices
\cite{JR:arxivdl}[Theorem \ref{teo:uni}].

The previous results are proved by showing their mirror versions in the
Lie algebra setting. This is a simpler task since we classified
Lie algebras of formal vector fields in \cite{JR:arxivdl}[Theorem 6].
Even if ${\mathfrak g}$ is an infinite dimensional complex vector space
the key invariant is the dimension of the finite dimensional
$\hat{K}_{n}$-vector space $\hat{K}_{n} \otimes_{\mathbb C} {\mathfrak g}$,
where $\hat{K}_{n}$ is the field of fractions of the ring of formal power series with
complex coefficients in $n$ variables.
The other main ingredient of the classification is also going to play a fundamental
role in this paper; it is
the field of formal meromorphic first integrals
\[ {\mathcal M} = \{ g \in \hat{K}_{n} : X (g) =0 \ \forall X \in {\mathfrak g} \} \]
of ${\mathfrak g}$.
More precisely, the classification relies on understanding the above invariants
for the Lie algebras in the derived series of ${\mathfrak g}$.

Of course in general $j^{1} G$ is not connected.
A naive approach to the general case is
splitting the problem in two simpler ones, restricting the
setting to some class of ``connected" subgroups of
${\rm Diff} ({\mathbb C}^{n},0)$ and then trying to reduce
the general case to the previous one by solving a linear problem.

Let us explain how to implement the aforementioned approach.
Given a subgroup $M$ of
$\mathrm{GL}(n,{\mathbb C})$ we denote by
$M_{u}$ its subset of unipotent matrices.
Notice that
if  $M$  is solvable then $M_{u}$ is a normal subgroup of $M$
(cf. Proposition \ref{pro:cr}).
We define the functions
\[ \psi_{1} (n) = \max \{ \ell (M/M_{u}) : M \ \mathrm{is} \ \mathrm{a} \ \mathrm{solvable} \ \mathrm{subgroup} \
\mathrm{of} \ \mathrm{GL}(n,{\mathbb C}) \}  \]
and
\[ \psi_{2} (n) = \max \{ \ell (G) : G \leq {\rm Diff} ({\mathbb C}^{n},0), \
\ell (G) < \infty \ \mathrm{and} \
j^{1} G =  (j^{1}G)_{u} \} . \]
Since $(j^{1}G)_{u}$ is a normal subgroup of $j^{1} G$, the
unipotent part
\[ G_{u}:= \{ \varphi \in G:  j^{1} \varphi \in (j^{1}G)_{u} \} \]
of $G$ is a normal subgroup of $G$ if $G$ is a solvable subgroup of
$\mathrm{Diff} ({\mathbb C}^{n},0)$.
It is clear that $\ell (G) \leq \ell (G/G_{u}) + \ell (G_{u})$.
In this way the problem splits in two simpler cases, namely a linear one
since $G/G_{u}$ is isomorphic to $j^{1}G /(j^{1}G)_{u}$
and the study of the connected group $G_{u}$ whose algebraic properties can
be interpreted in terms of a Lie algebra of (formal) nilpotent vector fields.
We obtain $\ell (G/G_{u})  \leq \psi_{1}(n)$ and
$\ell (G_{u}) \leq \psi_{2}(n)$ if $G$ is solvable.

The groups $G/G_{u}$ and $j^{1} G / (j^{1}G)_{u}$ are isomorphic to a completely reducible linear subgroup
of $\mathrm{GL}(n,{\mathbb C})$ (cf. section \ref{sec:linear}). In fact
$\psi_{1}$
coincides with the function $\sigma$ that determines the sharpest
upper bound for the derived length of completely reducible subgroups
of $\mathrm{GL}(n,{\mathbb C})$ (cf. Proposition \ref{pro:cr}).
The function $\sigma$ was calculated by Newman \cite{Newman}, for example
we have $\sigma(2)=4$, $\sigma (3)=\sigma (4) = \sigma (5)=5$.
The function $\psi_{2}(n)$ is equal to $2n-1$
\cite{JR:arxivdl}[Theorem \ref{teo:uni}].
The function $\sigma (n) + 2n-1$ can be considered as a baseline
upper bound for $\psi (n)$ in which the
reduction of the linear part of $G$ to a unipotent subgroup and
the reduction of the unipotent part $G_{u}$
of $G$ are treated as if they were independent processes.
As a consequence it is not surprising that this bound is not sharp.
Notice that the Main Theorem implies
$\psi (2) = (\sigma (2) + 2 \cdot 2 - 1) -2$ and
$\psi (n) = (\sigma (n) + 2n-1) -3$ for $n \in \{3,4,5\}$.
\subsection{Automorphisms of a Lie algebra of local vector fields}
In order to calculate $\psi$ we need to follow the evolution of the
unipotent parts of the derived subgroups of $G$. We can still profit
from the rigid structure of the Lie algebra ${\mathfrak g}$ of $G$
or the Lie algebra ${\mathfrak g}_{N}$ of $G_{u}$. Indeed we consider
the adjoint representation of $G$ on the Lie algebra ${\mathfrak g}$;
it induces a representation of $G$ on ${\mathfrak g}_{N}$.
The study of such action allows to improve the bound
$\sigma (n) + 2n-1$ to obtain the sharpest bound $2n+1$.


A prerequisite in our perspective is understanding the action of the
derived groups of $G$ on the field ${\mathcal M}$
of formal meromorphic first integrals of ${\mathfrak g}_{N}$.
On the one hand this action looks simpler since the transcendence degree
$\mathrm{trdeg}({\mathcal M}/{\mathbb C})$ of
the field extension ${\mathcal M}/{\mathbb C}$ is less than $n$ if ${\mathfrak g}_{N} \neq 0$
and hence we obtain an analogous problem in lower dimension.
On the other hand the group induced in the bijections of ${\mathcal M}$,
by the action of $G$ on ${\mathcal M}$,
is not necessarily isomorphic to a subgroup of
$\mathrm{Diff} ({\mathbb C}^{m},0)$ for some $m<n$ (or even to a subgroup
of the formal completion of $\mathrm{Diff} ({\mathbb C}^{m},0)$).
For instance $G$ can act on ${\mathcal M}$ as a
subgroup of M\"{o}bius transformations.
Therefore the study of the action of $G$ on ${\mathcal M}$
requires leaving the natural context of the problem and in particular the
local setting.
In spite of this we can describe the action of $G$ on ${\mathcal M}$
if $\mathrm{trdeg}({\mathcal M}/{\mathbb C}) \leq 1$.
The reason is the very simple structure of the fields of meromorphic first integrals
of codimension $1$ germs of singular foliations defined in a neighborhood of
the origin in ${\mathbb C}^{m}$ as a consequence of results of
Mattei-Moussu \cite{MaMo:Aen} and Cerveau-Mattei \cite{Ce-Ma:Ast}.
These results generalize to our formal setting.
Let us say a word about the hypothesis $n \leq 5$ on the dimension.
It is related to the aforementioned issues with fields of first integrals.
Indeed in the case $2 \leq n \leq 5$ we can reduce the proof of
$\psi (n) \leq 2n+1$ so it suffices to consider actions on fields of
(formal) meromorphic first integrals of transcendence degree at most $1$.
The other inequality, that completes the proof of the Main Theorem
is obtained by exhibiting an example
$G^{n}$ of solvable subgroup of $\mathrm{Diff}({\mathbb C}^{n},0)$
of derived length $2n+1$ for any $n \geq 2$.

\subsection{Remarks}
%
We associate to a solvable subgroup $G$
of $\mathrm{Diff} ({\mathbb C}^{n},0)$ some objects,
for instance the Lie algebra ${\mathfrak g}_{N}$ and the
induced representation of $G$ on ${\mathfrak g}_{N}$,
that make it easier to study its derived length.
The tools borrowed from the theory of linear algebraic groups allow
to keep an intrinsic point of view.
For example at no point we need to resort to the theory of normal forms
in order to prove $\psi (n) \leq 2n+1$.
Calculations at the jet level are not required either
since such structure is embedded in the construction of the pro-algebraic
closure of $G$.

The proof of $\psi(n) \leq 2n+1$ for $n \leq 5$ provides clues for
possible groups of maximal derived length and it
was useful to construct the actual examples.
Our techniques provide information about the algebraic properties of
a group of local complex analytic diffeomorphisms.
Their interest goes  beyond the
study of the derived length.
\section{Fundamental constructions}
In this section
we introduce the pro-algebraic subgroups of formal diffeomorphisms and their properties.
Although there are new bits, many of the notations and results are
borrowed from \cite{JR:arxivdl}.
\subsection{Formal diffeomorphisms}
\label{sec:fordif}
First of all we introduce some definitions.
\begin{defi}
We denote by $\hat{\mathcal O}_{n}$ the local ring
${\mathbb C} [[x_{1},\hdots,x_{n}]]$ of formal power series in $n$ variables with complex
coefficients.
We denote by ${\mathcal O}_{n}$ the subring ${\mathbb C} \{ x_{1},\hdots,x_{n} \}$ of
$\hat{\mathcal O}_{n}$ of convergent power series.
We denote by ${\mathfrak m}$ and $\hat{\mathfrak m}$ the maximal ideals of
${\mathcal O}_{n}$ and $\hat{\mathcal O}_{n}$ respectively.
\end{defi}
In our approach given a subgroup $G$ of $\mathrm{Diff} ({\mathbb C}^{n},0)$
we construct its pro-algebraic closure $\overline{G}^{(0)}$.
The latter group has the same derived length as $G$ and it is easier to handle
since we can apply the theory of linear algebraic groups.
Unfortunately $\overline{G}^{(0)}$ is not a subgroup of
$\mathrm{Diff} ({\mathbb C}^{n},0)$ in general. It is a subgroup of a formal
completion of $\mathrm{Diff} ({\mathbb C}^{n},0)$ whose elements are called
formal diffeomorphisms.
For completeness we introduce the construction of formal diffeomorphisms.

Given an element $\varphi \in \mathrm{Diff} ({\mathbb C}^{n},0)$ we consider its
action in the space of $k$-jets. More precisely we consider the
element $\varphi_{k} \in \mathrm{GL}({\mathfrak m}/{\mathfrak m}^{k+1})$ defined by
\begin{equation}
\label{eq:phik}
\begin{array}{ccc}
{\mathfrak m}/{\mathfrak m}^{k+1} &
\stackrel{\varphi_{k}}{\rightarrow} & {\mathfrak m}/{\mathfrak m}^{k+1} \\
g + {\mathfrak m}^{k+1} & \mapsto & g \circ \varphi + {\mathfrak m}^{k+1}
\end{array}
\end{equation}
where ${\mathfrak m}/{\mathfrak m}^{k+1}$ can be interpreted as a finite dimensional
complex vector space. In this point of view diffeomorphisms are interpreted
as operators acting on function spaces.
\begin{defi}
\label{def:dk}
We define $D_{k}=\{ \varphi_{k} : \varphi \in \mathrm{Diff} ({\mathbb C}^{n},0) \}$.
\end{defi}
The group $D_{k}$ can
be understood as an algebraic group of matrices by noticing that we have
\[ D_{k} = \{ \alpha \in \mathrm{GL}({\mathfrak m}/{\mathfrak m}^{k+1}) :
\alpha (g h) = \alpha (g) \alpha (h) \ \forall g,h \in  {\mathfrak m}/{\mathfrak m}^{k+1} \}  \]
and that fixed $g, h \in {\mathfrak m}/{\mathfrak m}^{k+1}$ the equation
$\alpha (g h) = \alpha (g) \alpha (h)$ is algebraic on the coefficients of $\alpha$.

It is clear that the natural projections $\pi_{k,l}: D_{k} \to D_{l}$ for $k \geq l$ define
a projective system and hence we can consider the projective limit $\varprojlim D_{k}$,
it is the so called group of formal diffeomorphisms.
Indeed given a sequence $(\phi_{k})_{k \geq 1}$ such that
$\phi_{k} \in D_{k}$ and $\pi_{k,l}(\phi_{k})= \phi_{l}$ for all $k \in {\mathbb N}$ and all $l \leq k$
the sequence $(\phi_{k}(x_{j} + {\mathfrak m}^{k+1}))_{k \geq 1}$ defines an element in
$\varprojlim {\mathfrak m}/{\mathfrak m}^{k+1}$
for every $1 \leq j \leq n$.
Since such a projective limit coincides with   $\hat{\mathfrak m}$,  we can consider
$(\phi_{k}(x_{j} + {\mathfrak m}^{k+1}))_{k \geq 1}$ as an element $g_{j}$ of $\hat{\mathfrak m}$.
Indeed the formal object
\[ \varphi (x_{1},\hdots,x_{n}):= (g_{1}(x_{1},\ldots,x_{n}) , \hdots , g_{n}(x_{1},\ldots,x_{n}) )
\in \hat{\mathfrak m}^{n} \]
behaves as a diffeomorphism since $\phi_{1} \in \mathrm{GL}({\mathfrak m}/{\mathfrak m}^{2})$ implies that
the first jet
\[ j^{1} \varphi = (j^{1} g_{1}, \hdots , j^{1} g_{n}) \]
belongs to $\mathrm{GL}(n, {\mathbb C})$
where given $f \in \hat{\mathfrak m}$ we denote by
$j^{k} f$ the unique polynomial of degree less or equal than $k$
such that $f - j^{k} f \in \hat{\mathfrak m}^{k+1}$.
Moreover
the $k$-jet $j^{k} \varphi = (j^{k} g_{1}, \hdots , j^{k} g_{n})$
belongs to $\mathrm{Diff} ({\mathbb C}^{n},0)$ for every $k \in {\mathbb N}$
by the inverse function theorem.
Thus we can interpret $\varphi$ as an element of
the formal completion of $\mathrm{Diff} ({\mathbb C}^{n},0)$ with respect to
the filtration ${\{ {\mathfrak m}^{k} \times \hdots \times {\mathfrak m}^{k} \}}_{k \in {\mathbb N}}$.
Moreover given
$f= \sum a_{j_{1},\ldots,j_{n}} x_{1}^{j_{1}} \ldots x_{n}^{j_{n}} \in
\hat{\mathfrak m}$ the composition
\[ f \circ \varphi : =
\sum_{j_{1}, \ldots, j_{n}} a_{j_{1},\ldots,j_{n}} g_{1}^{j_{1}} \ldots g_{n}^{j_{n}}
\]
is equal to the element
defined by $(\phi_{k}(f + {\mathfrak m}^{k+1}))_{k \geq 1}$ in $\hat{\mathfrak m}$.

We denote by $\diffh{}{n}$ the group of formal diffeomorphisms.
Resuming we can consider a formal diffeomorphism as either an element
$(\phi_{k})_{k \geq 1}$ in $\varprojlim D_{k}$ or as
an element $\varphi=(g_{1},\hdots,g_{n})$ of
$\hat{\mathfrak m} \times \hdots \times \hat{\mathfrak m}$ such that
$j^{1} \varphi \in \mathrm{GL}(n, {\mathbb C})$.
\begin{defi}
\label{def:u}
Let $\hat{\phi} \in \diffh{}{n}$. We say that $\hat{\phi}$ is unipotent if
$j^{1} \hat{\phi}$ is a unipotent linear isomorphism, i.e. if
$j^{1} \hat{\phi} - Id$ is nilpotent.
We denote by $\diffh{u}{n}$ the set of unipotent formal diffeomorphisms.
Given a subgroup $G$ of $\diffh{}{n}$ we denote by $G_{u}$ the set of
unipotent elements of $G$. We say that $G$
is unipotent if $G = G_{u}$.
\end{defi}
\begin{rem}
\label{rem:1forall}
An element $(\phi_{k})_{k \geq 1}$ of $\varprojlim D_{k}$ is unipotent if
and only if $\phi_{1}$ is unipotent. It would be more natural to define
$(\phi_{k})_{k \geq 1}$ as unipotent if $\phi_{k}$ is unipotent for any $k \in {\mathbb N}$.
It is a simple exercise that these definitions coincide.
Indeed the unipotent nature of $\phi_{1}$,
$\pi_{k,1}(\phi_{k})=\phi_{1}$ and the property
$\phi_{k}(gh)= \phi_{k}(g) \phi_{k}(h)$ for all $g,h \in {\mathfrak m}/{\mathfrak m}^{k+1}$
imply that $\phi_{k}$ is unipotent.
\end{rem}
\subsection{Formal vector fields}
We associate a Lie algebra ${\mathfrak g}$ to any subgroup $G$ of $\diffh{}{n}$.
Analogously as for diffeomorphisms ${\mathfrak g}$ is not in general
a Lie algebra of local vector fields
even if $G$ is a subgroup of $\mathrm{Diff} ({\mathbb C}^{n},0)$.

Let us denote by ${\mathfrak X} \cn{n}$ the Lie algebra of
(singular) local vector fields defined in the neighborhood of $0$ in ${\mathbb C}^{n}$.
An element $X$ of ${\mathfrak X} \cn{n}$
can be interpreted as a derivation of the ${\mathbb C}$-algebra ${\mathcal O}_{n}$
such that $X ({\mathfrak m}) \subset {\mathfrak m}$.
Naturally the Lie algebra $\hat{\mathfrak X} \cn{n}$ of formal vector fields in $n$ variables is
the set of derivations $X$ of $\hat{\mathcal O}_{n}$
such that $X (\hat{\mathfrak m}) \subset \hat{\mathfrak m}$.
A formal vector field $X \in \hat{\mathfrak X} \cn{n}$
is determined by $X(x_{1}), X(x_{2}), \hdots, X(x_{n})$. We obtain
\begin{equation}
\label{eq:fvf}
X = X(x_{1}) \frac{\partial}{\partial x_{1}} + \hdots + X(x_{n}) \frac{\partial}{\partial x_{n}} .
\end{equation}
\begin{defi}
\label{def:lk}
We define $L_{k}$ as the Lie algebra of derivations of the ${\mathbb C}$-algebra
${\mathfrak m}/ {\mathfrak m}^{k+1}$.
\end{defi}
Analogously as for formal diffeomorphisms the Lie algebra $\hat{\mathfrak X} \cn{n}$
can be understood as a projective limit $\varprojlim L_{k}$.
Given $X \in \hat{\mathfrak X} \cn{n}$ consider the element $(X_{k})_{k \geq 1}$ that
defines in $\varprojlim L_{k}$. Since $L_{k}$ is the Lie algebra of $D_{k}$
for any $k \in {\mathbb N}$, we obtain that $(\mathrm{exp}(X_{k}))_{k \geq 1}$ is a formal diffeomorphism
$\varphi$. Equivalently given $t \in {\mathbb C}$ the expression
\begin{equation}
\label{equ:expfor}
 {\rm exp} (t X) = \left({
\sum_{j=0}^{\infty}
\frac{t^{j}}{j!} X^{j}(x_{1}), \hdots, \sum_{j=0}^{\infty} \frac{t^{j}}{j!} X^{j}(x_{n})
}\right)
\end{equation}
defines the exponential of $t X$
where $X^{0}(f)=f$ and $X^{j+1}(f) = X(X^{j}(f))$ for all $f \in \hat{\mathcal O}_{n}$
and $j \geq 0$.
\begin{defi}
We say that a formal vector field $X \in \hat{\mathfrak X} \cn{n}$ is nilpotent if
$j^{1} X$ is a linear nilpotent vector field (cf. Equation (\ref{eq:fvf})).
We denote by $\hat{\mathfrak X}_{N} \cn{n}$ the set of nilpotent formal vector fields.
\end{defi}
\begin{rem}
\label{rem:1forallx}
Analogously as in Remark \ref{rem:1forall} an element $(X_{k})_{k \geq 1}$
of $\varprojlim L_{k}=\hat{\mathfrak X} \cn{n}$ is nilpotent if and only if $X_{k}$ is a nilpotent
linear map for any $k \in {\mathbb N}$.
\end{rem}
It is easier to deal with unipotent diffeomorphisms, instead of general ones,
since the formal properties of formal unipotent diffeomorphisms and formal
nilpotent vector fields are analogous.
\begin{pro}
\label{pro:ecmara}
(cf. \cite{Ecalle, MaRa:aen})
The image of $\hat{\mathfrak X}_{N} \cn{n}$ by the exponential map is equal to
$\diffh{u}{n}$ and
$\mathrm{exp}: \hat{\mathfrak X}_{N} \cn{n} \to \diffh{u}{n}$ is a bijection.
\end{pro}
\begin{defi}
Given $\varphi \in \diffh{u}{n}$ we define its infinitesimal generator
$\log \varphi$ as the unique element of
$\hat{\mathfrak X}_{N} \cn{n}$ such that $\varphi = \mathrm{exp} (\log \varphi)$.
We define the $1$-parameter group $(\varphi^{t})_{t \in {\mathbb C}}$ by
$\varphi^{t} = \mathrm{exp} (t \log \varphi)$.
\end{defi}
Generically the infinitesimal generator of a local diffeomorphism
is a divergent vector field.

The properties of a unipotent diffeomorphism can be expressed in terms of
its infinitesimal generator. We provide some simple examples.
\begin{lem}
\label{lem:pol}
Let $\varphi \in \diffh{u}{n}$.
\begin{itemize}
\item Let $X \in \hat{\mathfrak X} \cn{n}$. If $\varphi_{*} X = X$ then $[\log \varphi, X]=0$.
\item Let $f \in \hat{\mathcal O}_{n}$. If $f \circ \varphi =f$ then $(\log \varphi)(f)=0$.
\end{itemize}
\end{lem}
\begin{proof}
Suppose $\varphi_{*} X = X$.
It is clear that $(\varphi^{t})^{*} X = X$ for any $t \in {\mathbb Z}$.
Since $\log \varphi$ is nilpotent the expression
$(\varphi^{t})^{*} X = {\rm exp}(t \log \varphi)^{*} X$ is polynomial
in $t$. We deduce $(\varphi^{t})^{*} X = X$ for any
$t \in {\mathbb C}$ since a polynomial vanishing in ${\mathbb Z}$ is identically $0$.
We obtain
\[   [\log \varphi, X] = \lim_{t \to 0} \frac{(\varphi^{t})^{*} X  - X}{t} = 0 . \]
The proof of the second property is analogous.
\end{proof}
\subsection{Pro-algebraic groups of local diffeomorphisms}
\label{sec:forgroup}
In this section we define pro-algebraic subgroups of
$\mathrm{Diff} ({\mathbb C}^{n},0)$ (or $\diffh{}{n}$).
Moreover given a subgroup $G$ of $\diffh{}{n}$ we define the minimal
pro-algebraic subgroup containing $G$, the so called pro-algebraic
closure.
Along the way we explain how to adapt some concepts of Lie theory to
pro-algebraic groups in our
infinite dimensional setting. For instance we define the Lie algebra
of a subgroup of $\diffh{}{n}$, the connected component of
the identity of a pro-algebraic subgroup of $\diffh{}{n}$...
\begin{defi}
Let us define the Krull topology in $\diffh{}{n}$.
Let $k \in {\mathbb N}$ and $\varphi \in \diffh{}{n}$;
we denote
\[ U_{k, \varphi} = \{ \eta \in \diffh{}{n}: j^{k} \eta = j^{k} \varphi \}. \]
The family
$\{ U_{k, \varphi} \}$ where $k$ varies in ${\mathbb N}$ and $\varphi$ varies in
$\diffh{}{n}$
is a base of neighborhoods for the Krull topology in $\diffh{}{n}$. A sequence
$(\varphi_{j})_{j \geq 1}$ converges to $\varphi$ if given any $k \in {\mathbb N}$
there exists $j_{0} \in {\mathbb N}$ such that $j^{k} \varphi_{j} = j^{k} \varphi$
for any $j \geq j_{0}$. The definition  of the Krull topology for formal vector fields
is analogous.
\end{defi}
Our next goal is defining the pro-algebraic closure of a subgroup of $\diffh{}{n}$.
It is the analogue of the algebraic closure of a linear group in the context of
formal diffeomorphisms.
\begin{defi}
\label{def:gk}
Let $G$ be a subgroup of $\diffh{}{n}$.
We define $G_{k}$ as the smallest algebraic subgroup
of $D_{k}$ containing $\{ \varphi_{k} : \varphi \in G \}$ (cf. Equation (\ref{eq:phik})).
\end{defi}
\begin{defi}
Let $G$ be a subgroup of $\diffh{}{n}$.
We define
$\overline{G}^{(0)}$ as $\varprojlim_{k \in {\mathbb N}} G_{k}$, more precisely
$\overline{G}^{(0)}$ is the subgroup of $\diffh{}{n}$ defined by
\[ \overline{G}^{(0)} = \{ \varphi  \in \diffh{}{n} : \varphi_{k} \in G_{k} \ \forall k \in {\mathbb N} \} . \]
We say that $\overline{G}^{(0)}$ is the pro-algebraic closure of $G$.
We say that $G$ is {\it pro-algebraic} if $G = \overline{G}^{(0)}$.
\end{defi}
\begin{rem}
\label{rem:pacck}
The pro-algebraic closure
of a subgroup
of $\diffh{}{n}$
is closed in the Krull topology by construction.
\end{rem}
The next two results are technical lemmas that we use to
characterize the pro-algebraic subgroups of $\diffh{}{n}$.
\begin{lem}
\label{lem:cons}
Let $H_{k}$ be an algebraic subgroup of $D_{k}$ for $k \in {\mathbb N}$.
Suppose that $\pi_{l,k}(H_{l}) = H_{k}$ for all $l \geq k \geq 1$.
Then $\varprojlim_{k \in {\mathbb N}} H_{k}$ is a pro-algebraic subgroup of
$\diffh{}{n}$. Moreover the natural map $\varprojlim H_{j}  \to H_{k}$
is surjective for any $k \in {\mathbb N}$.
\end{lem}
\begin{proof}
The inverse limit $\varprojlim H_{k}$ is contained in
$\diffh{}{n} = \varprojlim D_{k}$.

An inverse system $(S_{k})_{k \in {\mathbb N}}$ of non-empty sets and surjective maps
indexed by the natural numbers satisfies that the natural projections
$\varprojlim_{j \in {\mathbb N}} S_{j} \to S_{k}$ are surjective for any $k \in {\mathbb N}$.
Since $(\pi_{l,k})_{|H_{l}}: H_{l}  \to H_{k}$ is surjective for $l \geq k \geq 1$,
the natural map $\varprojlim_{j \in {\mathbb N}} H_{j} \to H_{k}$ is surjective
for any $k \in {\mathbb N}$. In particular we obtain
$\{ \varphi_{k} : \varphi \in \varprojlim H_{j} \} = H_{k}$ for $k \in {\mathbb N}$
and then $\overline{(\varprojlim H_{k})}^{(0)} = \varprojlim H_{k}$.
\end{proof}
\begin{lem}
\label{lem:surpro}
Let $G$ be a subgroup of $\diffh{}{n}$. Then we obtain
$\pi_{l,k}(G_{l})=G_{k}$ for all $l \geq k \geq 1$.
Moreover the natural map $\varprojlim G_{j} \to G_{k}$ is surjective for
any $k \in {\mathbb N}$.
\end{lem}
\begin{proof}
The map $\pi_{l,k}: D_{l}  \to D_{k}$ is a surjective morphism of algebraic groups
for $l \geq k$.
Moreover the image by $\pi_{l,k}$ of the smallest algebraic group of
$\mathrm{GL}({\mathfrak m}/{\mathfrak m}^{l+1})$ containing $\{ \varphi_{l} : \varphi \in G \}$
is the smallest algebraic group of
$\mathrm{GL}({\mathfrak m}/{\mathfrak m}^{k+1})$ containing
$\{ \varphi_{k} : \varphi \in G \} = \pi_{l,k}(\{ \varphi_{l} : \varphi \in G \})$
(cf.  \cite{Borel}[2.1 (f), p. 57]).
Hence we have $\pi_{l,k} (G_{l})=G_{k}$ if $l \geq k$.
The last statement is a consequence of Lemma \ref{lem:cons}.
\end{proof}
We provide two characterizations of pro-algebraic groups in next proposition.
\begin{pro}
\label{pro:char}
Let $G$ be a subgroup of $\diffh{}{n}$. Then the following conditions are equivalent:
\begin{enumerate}
\item $G$ is pro-algebraic.
\item $\{ \varphi_{k} : \varphi \in G \}$ is an algebraic matrix group for any $k \in {\mathbb N}$
and $G$ is closed in the Krull topology.
\item $G$ is of the form $\varprojlim_{k \in {\mathbb N}} H_{k}$ where $H_{k}$ is an algebraic subgroup of $D_{k}$
and $\pi_{l,k} (H_{l}) \subset H_{k}$ for all $l \geq k \geq  1$.
\end{enumerate}
\end{pro}
\begin{proof}
Let us prove $(1) \implies (2)$.  Suppose $G = \overline{G}^{(0)}$.
We obtain
$\{ \varphi_{k} : \varphi \in G \} = G_{k}$ by Lemma \ref{lem:surpro}. Moreover since $ \overline{G}^{(0)}$
is closed in the Krull topology by construction, $G$ is closed in the Krull topology.

Let us show $(2) \implies (1)$. The group $G_{k}$ is equal to
$\{ \varphi_{k} : \varphi \in G \}$ by hypothesis for any $k \in {\mathbb N}$ .
We claim $\overline{G}^{(0)} \subset G$. Indeed given $\varphi \in \overline{G}^{(0)}$ and $k \in {\mathbb N}$
there exists $\eta(k) \in G$ such that $\varphi_{k} =(\eta(k))_{k}$ for any $k \in {\mathbb N}$
since $\overline{G}^{(0)} = \varprojlim \{ \varphi_{k} : \varphi \in G \}$.
In particular $\varphi = \lim_{k \to \infty} \eta (k)$ where the limit is considered in the Krull topology.
Since $G$ is closed in the Krull topology, we obtain $\varphi \in G$.
The inclusion $\overline{G}^{(0)} \subset G$ implies $G=\overline{G}^{(0)}$ and hence
$G$ is pro-algebraic.
Moreover we obtain $G=\overline{G}^{(0)} = \varprojlim G_{k}$
and then $G$ is the form in item $(3)$ by Lemma \ref{lem:surpro}.
We just proved $(2) \implies (3)$.

Finally let us prove $(3) \implies (1)$.
We define $H_{l,k} = \pi_{l,k}(H_{l})$ for $l \geq k \geq 1$.
The group $H_{l,k}$ is algebraic since it is the image of an algebraic group
by a morphism of algebraic groups.
Since $\pi_{l',k} = \pi_{l,k} \circ \pi_{l',l}$ for $l' \geq l \geq k \geq 1$,
the sequence $(H_{l,k})_{l \geq k}$ is decreasing for any $k \in {\mathbb N}$.
The sequence stabilizes by the noetherianity of the ring of regular functions of
an affine algebraic variety. We denote $K_{k}= \cap_{l \geq k} H_{l,k}$.
Given $l \geq k \geq 1$ we consider $l' \geq l$ such that
$K_{l}= H_{l',l}$ and $K_{k}= H_{l',k}$.
Since $\pi_{l',k} = \pi_{l,k} \circ \pi_{l',l}$, we deduce $\pi_{l,k}(K_{l})=K_{k}$
for all $l \geq k \geq 1$.
The construction implies $\varprojlim K_{k} = \varprojlim H_{k}$.
Thus $\varprojlim H_{k}$ is pro-algebraic by Lemma \ref{lem:cons}.
\end{proof}
\begin{rem}
Proposition \ref{pro:char} is very useful to show that certain groups are pro-algebraic.
For example consider a family ${\{G_{j}\}}_{j \in J}$ of pro-algebraic subgroups of $\diffh{}{n}$.
Let us see that $\cap_{j \in J} G_{j}$ is pro-algebraic.
We have
\[ \pi_{l,k} (\cap_{j \in J} (G_{j})_{l}) \subset \cap_{j \in J} (G_{j})_{k} \ \forall l \geq k \geq 1
\ \mathrm{and} \ \cap_{j \in J} G_{j} = \varprojlim \cap_{j \in J} (G_{j})_{k}. \]
Since the intersection of algebraic matrix groups is an algebraic group, the group
$\cap_{j \in J} G_{j}$ is pro-algebraic by item (3) of Proposition \ref{pro:char}.
\end{rem}
\begin{rem}
\label{rem:invpro}
Invariance properties typically define pro-algebraic groups.
Item $(3)$ of Proposition \ref{pro:char} provides an easy way
of proving such property.

For instance consider $f \in \hat{\mathcal O}_{n}$ and
$G = \{ \varphi \in \diffh{}{n} : f \circ \varphi \equiv f \}$.
We define
\[ H_{k}= \{ A \in D_{k} : A (f +{\mathfrak m}^{k+1}) = f +{\mathfrak m}^{k+1} \}; \]
it is clear that $H_{k}$ is an algebraic subgroup of $D_{k}$ and $\pi_{l,k}(H_{l}) \subset H_{k}$
for all $l \geq k \geq 1$.
Moreover since $G= \varprojlim H_{k}$, $G$ is a pro-algebraic group by Proposition \ref{pro:char}.
The power of item $(3)$ of Proposition \ref{pro:char}  is that
in order to show that $G$ is pro-algebraic we do not need to
find $\{ \varphi_{k} : \varphi \in G \}$ explicitly; in particular we could have
$\{ \varphi_{k} : \varphi \in G \} \subsetneq H_{k}$. Moreover, it allows us to exploit that
a pro-algebraic group can be expressed in several ways as an inverse limit of algebraic
groups.
\end{rem}
The group $\overline{G}^{(0)}$
is a projective limit
of algebraic groups and
closed in the Krull topology by definition.
%
Since $G_{k}$ is an algebraic group of matrices and in particular a Lie group,
we can define the conected component $G_{k,0}$ of the identity in $G_{k}$.
We also consider the set $G_{k,u}$ of unipotent elements
of $G_{k}$.
\begin{defi}
\label{def:gu}
Let $G$ be a subgroup of $\diffh{}{n}$. We define
\[ \overline{G}_{0}^{(0)} = \{ \varphi  \in \diffh{}{n} : \varphi_{k} \in G_{k,0} \ \forall k \in {\mathbb N} \} \]
and
\[ \overline{G}_{u}^{(0)} = \{ \varphi  \in \diffh{}{n} : \varphi_{k} \in G_{k,u} \ \forall k \in {\mathbb N} \}. \]
\end{defi}
Let us show that membership in $\overline{G}_{0}^{(0)}$ can be checked out on the linear part
and that $\overline{G}_{0}^{(0)}$ is pro-algebraic.
\begin{pro}
\label{pro:ccipa}
Let $G$ be a  subgroup of $\diffh{}{n}$. Then we have
$\overline{G}_{0}^{(0)} = \{ \varphi  \in \overline{G}^{(0)} : \varphi_{1} \in G_{1,0} \}$.
Moreover $\overline{G}_{0}^{(0)}$ is pro-algebraic.
\end{pro}
\begin{proof}
Lemma 3 of \cite{JR:arxivdl} implies
$\pi_{l,k}^{-1} (G_{k,0})= G_{l,0}$ and $\pi_{l,k} (G_{l,0})= G_{k,0}$
for all $l \geq k \geq 1$.
We deduce $\overline{G}_{0}^{(0)} = \{ \varphi  \in \overline{G}^{(0)} : \varphi_{1} \in G_{1,0} \}$.

Since $\overline{G}_{0}^{(0)} = \varprojlim_{k \in {\mathbb N}} G_{k,0}$ and
$\pi_{l,k}: G_{l,0} \to G_{k,0}$ is surjective for all $l \geq k \geq 1$, the
group $\overline{G}_{0}^{(0)}$ is pro-algebraic by Lemma \ref{lem:cons}.
%
\end{proof}
We prove next that $\overline{G}_{u}^{(0)}$ is a pro-algebraic group if $G$ is solvable.
\begin{lem}
\label{lem:solaux}
Let $G$ be a solvable subgroup of $\diffh{}{n}$.
Then
\begin{itemize}
\item $\ell (\overline{G}^{(0)})=\ell (G)$.
\item $G_{k,u} \subset G_{k,0}$ for any $k \in {\mathbb N}$.
\item ${\rm exp}(t \log \varphi) \in \overline{G}_{u}^{(0)}$ for all
$\varphi \in \overline{G}_{u}^{(0)}$ and $t \in {\mathbb C}$.
\item $G_{k,u}$ is a normal connected algebraic subgroup of
the group $G_{k}$ for any $k \in {\mathbb N}$.
\item $\overline{G}_{u}^{(0)}$
is a pro-algebraic normal subgroup of $\overline{G}^{(0)}$.
Moreover $\overline{G}^{(0)} = \overline{G}_{u}^{(0)}$ if $G$ is unipotent.
\end{itemize}
\end{lem}
\begin{proof}
The three first properties are proved in Lemma 1 of \cite{JR:arxivdl}.
Indeed these properties hold true even if $G$ is not solvable.

Since $G_{k,0}$ is a connected solvable algebraic group, its
subset $G_{k,u}$ of unipotent elements is a connected normal closed subgroup of
$G_{k,0}$
(cf.  \cite{Borel}[Theorem 10.6, p. 137]).
We obtain that  $G_{k,u}$ is a subgroup (obviously normal) of $G_{k}$
for any $k \in {\mathbb N}$.

The previous property implies that $\overline{G}_{u}^{(0)}$ is a normal subgroup of
$\overline{G}^{(0)}$.
We have $\pi_{l,k}^{-1} (G_{k,u})= G_{l,u}$ for all $l \geq k \geq 1$ by Remark \ref{rem:1forall}.
Since $\pi_{l,k}(G_{l})= G_{k}$ by Lemma \ref{lem:surpro}, hence $\pi_{l,k}(G_{l,u})= G_{k,u}$ for all $l \geq k \geq 1$.
Therefore $\overline{G}_{u}^{(0)} = \varprojlim G_{k,u}$ is pro-algebraic by
Lemma \ref{lem:cons}.
%
The group $\overline{G}^{(0)}$ is unipotent if $G$ is unipotent by
Lemma 4 of \cite{JR:arxivdl}.
\end{proof}
\begin{rem}
\label{rem:fincod}
Let $G$ be a solvable subgroup of $\mathrm{Diff}({\mathbb C}^{n},0)$.
A useful feature of the groups
$\overline{G}_{0}^{(0)}$ and $\overline{G}_{u}^{(0)}$
is that in spite of being infinite dimensional in general,
$\overline{G}_{u}^{(0)}$ has finite codimension and
$\overline{G}_{0}^{(0)}$ has codimension $0$ in $\overline{G}^{(0)}$.
More precisely, the groups
$\overline{G}^{(0)}/\overline{G}_{u}^{(0)}$ and $\overline{G}^{(0)}/\overline{G}_{0}^{(0)}$
are isomorphic to $G_{1}/G_{1,u}$ and $G_{1}/G_{1,0}$ respectively by Remark \ref{rem:1forall} and
Proposition \ref{pro:ccipa}.
In particular $\overline{G}^{(0)}/\overline{G}_{0}^{(0)}$ is a finite group.
\end{rem}
We can associate Lie algebras to pro-algebraic groups.
\begin{defi}
\label{def:liealg}
Let $G$ be a subgroup of $\diffh{}{n}$.
We define the set
\[ {\mathfrak g} = \{ X \in \hat{\mathfrak X} \cn{n} : X_{k} \in {\mathfrak g}_{k} \
\forall k \in {\mathbb N} \} \]
where ${\mathfrak g}_{k}$ is  the Lie algebra of
$G_{k}$. We say that ${\mathfrak g}$ is the Lie algebra of the pro-algebraic group
$\overline{G}^{(0)}$.

Suppose $G$ is solvable, we define
\[ {\mathfrak g}_{N} = \{ X \in \hat{\mathfrak X} \cn{n} :
 X_{k} \in {\mathfrak g}_{k,u} \
\forall k \in {\mathbb N} \} \]
where   ${\mathfrak g}_{k,u}$ is the Lie algebra of
$G_{k,u}$.
\end{defi}
Notice that by definition ${\mathfrak g}_{N}$ is the Lie algebra of the pro-algebraic group
$\overline{G}_{u}^{(0)}$.

The Lie algebra of a pro-algebraic subgroup of $\diffh{}{n}$
shares analogous properties with the finite dimensional case.
\begin{pro}
\label{pro:lie} \cite{JR:arxivdl}[Proposition 2]
Let $G \subset \diffh{}{n}$ be a   group.
Then ${\mathfrak g}$   is
equal to
$\{ X \in \hat{\mathfrak X} \cn{n} : \mathrm{exp}(t X) \in \overline{G}^{(0)} \ \forall t \in {\mathbb C} \}$
and
$\overline{G}_{0}^{(0)}$ is generated by the set
$\{ \mathrm{exp}(X) : X \in {\mathfrak g} \}$.
Moreover if $G$ is unipotent (cf. Definition \ref{def:u}) the map
\[ {\rm exp}: {\mathfrak g} \to \overline{G}^{(0)} \]
is a bijection and ${\mathfrak g}$ is a Lie algebra of nilpotent formal vector fields.
\end{pro}
\begin{rem}
The term ``connected component of the identity of $\overline{G}^{(0)}$" for
$\overline{G}_{0}^{(0)}$ is completely
justified. On the one hand $\overline{G}^{(0)}/\overline{G}_{0}^{(0)}$ is a
finite group by Remark \ref{rem:fincod}.
On the other hand every element $\varphi$ of $\overline{G}_{0}^{(0)}$
is of the form $\mathrm{exp}(X_{1}) \circ \hdots \circ \mathrm{exp}(X_{k})$
where $X_{1},\hdots, X_{k} \in {\mathfrak g}$ by Proposition \ref{pro:lie}.
Hence $\mathrm{exp}(t X_{1}) \circ \hdots \circ \mathrm{exp}(t X_{k})$
describes a path connecting the identity with $\varphi$
in $\overline{G}_{0}^{(0)}$ when $t$ varies in $[0,1]$.
\end{rem}
\begin{cor}
\label{cor:lie}
Let $G \subset \diffh{}{n}$ be a solvable group.
Then ${\mathfrak g}_{N}$ is a complex Lie algebra of nilpotent formal vector fields such that
\[ {\rm exp}: {\mathfrak g}_{N} \to \overline{G}_{u}^{(0)} \]
is a bijection.
Moreover ${\mathfrak g}_{N}$ is closed in the Krull topology.
\end{cor}
\begin{proof}
Denote $H=\overline{G}_{u}^{(0)}$.  Then $H$ is a solvable unipotent pro-algebraic
group by Lemma \ref{lem:solaux}.
Since ${\mathfrak g}_{N}$ is the Lie algebra of $H$,
the result is a consequence of Proposition \ref{pro:lie}.
Notice that ${\mathfrak g}_{N}$ is closed in
the Krull topology by construction.
\end{proof}
Let us present a method to construct unipotent pro-algebraic groups
as exponentials of Lie algebras of formal nilpotent vector fields.
\begin{pro}
\label{pro:consag}
Let ${\mathfrak g}$ be a Lie subalgebra of $\hat{\mathfrak X}_{N} \cn{n}$ that is closed in
the Krull topology. Then the set $\mathrm{exp} ({\mathfrak g}) := \{ \mathrm{exp} (X) : X \in {\mathfrak g} \}$
is a pro-algebraic unipotent
subgroup of $\diffh{}{n}$ whose Lie algebra is ${\mathfrak g}$.
\end{pro}
\begin{proof}
Every element $X$ of $\hat{\mathfrak X} \cn{n}$ induces a derivation of the
${\mathbb C}$-algebra ${\mathfrak m}/{\mathfrak m}^{k+1}$, i.e. an element of $L_{k}$.
Let ${\mathfrak g}_{k}$ be the Lie subalgebra of the Lie algebra $L_{k}$
induced by ${\mathfrak g}$.
The elements of the Lie algebra ${\mathfrak g}_{k}$
are nilpotent for any $k \in {\mathbb N}$ by Remark \ref{rem:1forallx}.
Hence ${\mathfrak g}_{k}$ is algebraic, i.e. it is the Lie algebra of a linear algebraic group
(cf. \cite{Borel}[p. 106 and Corollary 7.7, p. 108]). Let $G_{k}$ be the connected linear algebraic
subgroup of $\mathrm{GL}({\mathfrak m}/{\mathfrak m}^{k+1})$ whose Lie algebra is
${\mathfrak g}_{k}$ for $k \in {\mathbb N}$.
The Lie algebra ${\mathfrak g}_{k}$ is contained in the Lie algebra $L_{k}$ of
$D_{k}$. Thus $G_{k}$ is contained in $D_{k}$.

Since all elements of ${\mathfrak g}_{k}$ are nilpotent,
${\mathfrak g}_{k}$ is upper triangular up to a linear change of coordinates by
Engel's theorem (cf. \cite{Serre.Lie}[chapter V, section 2, p. 33]).
In such coordinates the diagonal coefficients of any element of ${\mathfrak g}_{k}$
vanish. In particular $\mathrm{exp} ({\mathfrak g}_{k})$
is a set of upper triangular matrices whose diagonal elements are all equal to $1$.
Therefore the group $G_{k} = \langle \mathrm{exp} ({\mathfrak g}_{k}) \rangle$
is upper triangular and all its elements are unipotent.
We deduce that $\mathrm{exp} : {\mathfrak g}_{k} \to G_{k}$ is a bijection
and in particular $G_{k} = \mathrm{exp} ({\mathfrak g}_{k})$ for any
$k \in {\mathbb N}$.
Since ${\mathfrak g}$ is closed in the Krull topology, we obtain
$\mathrm{exp} ({\mathfrak g}) = \varprojlim  \mathrm{exp} ({\mathfrak g}_{k})$.
Thus $\mathrm{exp} ({\mathfrak g})$ is a group that is closed in the Krull topology.
Moreover, since $\mathrm{exp} ({\mathfrak g}_{k})$ is algebraic for any $k \in {\mathbb N}$,
the group $\overline{\mathrm{exp} ({\mathfrak g})}^{(0)}$ is equal to $\mathrm{exp} ({\mathfrak g})$
and its Lie algebra is equal to ${\mathfrak g} = \varprojlim  {\mathfrak g}_{k}$
by construction.
\end{proof}
We already have two methods to construct pro-algebraic groups, namely
as groups defined by invariance properties (Remark \ref{rem:invpro}) or
as exponentials of Lie algebras of formal nilpotent vector fields.

The previous proposition makes explicit the Lie correspondence in the context of
unipotent pro-algebraic groups of formal diffeomorphisms.
\begin{pro}
\label{pro:liecor}
The map ${\mathfrak g} \to \mathrm{exp} ({\mathfrak g})$ is a bijection from
Lie subalgebras of $\hat{\mathfrak X}_{N} ({\mathbb C}^{n},0)$ that are closed in the Krull topology
to unipotent pro-algebraic subgroups of $\diffh{}{n}$.
\end{pro}
\begin{proof}
The map is well-defined by Proposition \ref{pro:consag}.
It is surjective by Proposition \ref{pro:lie}.
It is injective by Proposition \ref{pro:ecmara}.
\end{proof}
Let us give a characterization of the pro-algebraic closure of a unipotent group
in terms of one-parameter groups.
\begin{pro}
\label{pro:charupro}
Let $G$ be a unipotent subgroup of $\diffh{}{n}$. Then $\overline{G}^{(0)}$
is equal to the closure in the Krull topology of $\langle \cup_{\varphi \in G} \{ \varphi^{t} : t \in {\mathbb C} \} \rangle$.
\end{pro}
\begin{proof}
Let $H$ be the closure in the Krull topology of $\langle \cup_{\varphi \in G} \{ \varphi^{t} : t \in {\mathbb C} \} \rangle$.
The group $H$ is contained in $\overline{G}^{(0)}$ by Proposition \ref{pro:liecor}.

We denote
$C_{k}= \{ \varphi_{k} : \varphi \in G \}$ for any $k \in {\mathbb N}$.
It is unipotent, i.e. it consists of unipotent elements by Remark \ref{rem:1forall}.
Notice that given a unipotent $A \in \mathrm{GL}({\mathfrak m}/{\mathfrak m}^{k+1})$
there exists a unique nilpotent $\log A \in \mathrm{End}({\mathfrak m}/{\mathfrak m}^{k+1})$ such
that $\mathrm{exp}(\log A)=A$. Moreover if $A= \varphi_{k}$ for some $\varphi \in G$ then
$\log A = (\log \varphi)_{k}$.
Thus we obtain
\[ \cup_{A \in C_{k}} \{ \mathrm{exp} (t \log A) : t \in {\mathbb C} \} =
\cup_{\varphi \in G} \{ \varphi_{k}^{t} : t \in {\mathbb C} \} \]
A theorem of Chevalley implies that a group generated by a union of
connected algebraic groups is algebraic (cf. \cite{Borel}[section I.2.2, p. 57]);
in particular $\langle \cup_{A \in C_{k}} \{ \mathrm{exp} (t \log A) : t \in {\mathbb C} \} \rangle$
and then $\langle \cup_{\varphi \in G} \{ \varphi_{k}^{t} : t \in {\mathbb C} \} \rangle$ are algebraic.
Therefore $\{ \varphi_{k} : \varphi \in H \}$ is algebraic for any $k \in {\mathbb N}$.
Since $H$ is closed in the Krull topology by construction, $H$ is pro-algebraic by
Proposition \ref{pro:char}. The property
$G \subset H \subset \overline{G}^{(0)}$ implies $\overline{G}^{(0)} \subset H \subset \overline{G}^{(0)}$
and then $\overline{G}^{(0)} = H$.
\end{proof}
Proposition \ref{pro:charupro} provides an alternate characterization of unipotent pro-algebraic subgroups
of $\diffh{}{n}$. 
\begin{cor}
\label{cor:charupro}
Let $G$ be a unipotent subgroup of $\diffh{}{n}$. Then $G$ is pro-algebraic if and only if
$G$ is closed in the Krull topology and the one-parameter group $\{ \varphi^{t} : t \in {\mathbb C} \}$
is contained in $G$ for any $\varphi \in G$. \end{cor}
\begin{proof}
Let $H$ be the closure in the Krull topology of $\langle \cup_{\varphi \in G} \{ \varphi^{t} : t \in {\mathbb C} \} \rangle$.
We have $G \subset H =  \overline{G}^{(0)}$ by Proposition \ref{pro:liecor}.

Suppose that $G$ is pro-algebraic.  Then $H$ is equal to $G$ and it is clear that 
$\cup_{\varphi \in G} \{ \varphi^{t} : t \in {\mathbb C} \} \subset G$ and that $G$ is closed
in the Krull topology.

Suppose that $\cup_{\varphi \in G} \{ \varphi^{t} : t \in {\mathbb C} \} \subset G$ and that 
$G$ is closed in the Krull topology. 
Hence $H$ is contained in $G$. Since $G \subset H$, we deduce $H=G$. 
Finally we remark that $H= \overline{G}^{(0)}$ is pro-algebraic.
\end{proof}
\begin{rem}
\label{rem:exaclo}
As an example
let us see that the pro-algebraic closure of a cyclic group $\langle \varphi \rangle$ for
$\varphi \in \diffh{u}{n}$ is equal to $\{ \varphi^{t} : t \in {\mathbb C} \}$.
The group $\overline{\langle \varphi \rangle}^{(0)}$ contains
$\{ \varphi^{t} : t \in {\mathbb C} \}$ by the third property of Lemma \ref{lem:solaux}.
It suffices to show that $\{ \varphi^{t} : t \in {\mathbb C} \}$ is pro-algebraic.
This is a consequence of Proposition \ref{pro:consag} since $\{ \varphi^{t} : t \in {\mathbb C} \}$
is the exponential of the Lie algebra $\langle \log \varphi \rangle$ of dimension less or equal than $1$
generated by $\log \varphi$ which is closed in the Krull topology.
It is also an easy consequence of Corollary \ref{cor:charupro} since
$\{ \varphi^{t} : t \in {\mathbb C} \}$ is closed in the Krull topology.
%
%
\end{rem}
\begin{rem}
Remark \ref{rem:exaclo} provides conceptual explanations for the extension of invariance properties
to the one parameter group of $\varphi \in \diffh{u}{n}$.
For example suppose $f \circ \varphi \equiv f$ for some $f \in \hat{\mathcal O}_{n}$.
Since the group $G = \{ \eta \in \diffh{}{n} : f \circ \eta \equiv f \}$  is pro-algebraic by
Remark \ref{rem:invpro}, it contains
$ \{ \varphi^{t} : t \in {\mathbb C} \} = \overline{\langle \varphi \rangle}^{(0)}$ and hence
we deduce $(\log \varphi)(f) \equiv 0$.
\end{rem}
\subsection{Derived series}
The goal of this paper is obtaining sharp bounds for the derived length
of solvable groups of local diffeomorphisms.
Let us introduce some definitions for the sake of clarity.
\begin{defi}
\label{def:com}
Let $G$ be a group. We denote by $[\alpha,\beta]$ the commutator
$\alpha \beta \alpha^{-1} \beta^{-1}$ of $\alpha$ and $\beta$.
Given subgroups $H$, $L$ of $G$ we define $[H,L]$ as the subgroup
generated by the elements of the form $[h,l]$ for all $h \in H$ and $l \in L$.
We denote $G^{(0)}=G$.
We define by induction the $j$-derived group $G^{(j)}:= [G^{(j-1)}, G^{(j-1)}]$
of $G$ for $j >0$ (or just the derived group if $j=1$).
\end{defi}
\begin{defi}
We can define the derived series of a complex Lie algebra ${\mathfrak g}$
analogously. The subgroup generated by the commutators is replaced
by the Lie subalgebra generated by the Lie brackets.
We define ${\mathfrak g}^{(0)}:={\mathfrak g}$ and
${\mathfrak g}^{(j)}:=[{\mathfrak g}^{(j-1)}, {\mathfrak g}^{(j-1)}]$
for $j>0$.
\end{defi}
\begin{defi}
\label{def:length}
Let $G$ be a group (resp. Lie algebra).
We define the derived length
\[ \ell (G) = \min \{ k \in {\mathbb N} \cup \{0\} : G^{(k)} \ \mathrm{is} \ \mathrm{trivial} \} \]
where $\min \emptyset = \infty$ by convention. We say that $G$ is {\it solvable} if
$\ell (G) < \infty$.
\end{defi}
The next definitions were introduced in \cite{JR:arxivdl}.
\begin{defi}
\label{def:clser}
Let $G$ be a subgroup of $\diffh{}{n}$.
By induction we define the $j$-closed derived group $\overline{G}^{(j)}$ of $G$
as the closure in the Krull topology of $[\overline{G}^{(j-1)}, \overline{G}^{(j-1)}]$
for any $j \in {\mathbb N}$.
\end{defi}
Next results are straightforward.
\begin{lem}
\label{lem:normalg}
Let $G$ be a subgroup of $\diffh{}{n}$.
Then $\overline{G}^{(j)}$  is the closure
in the Krull topology of the $j$-derived group
of $\overline{G}^{(0)}$  for any $j \in {\mathbb N}$. Moreover, the series
$\hdots \triangleleft \overline{G}^{(m)}
\triangleleft \hdots \triangleleft \overline{G}^{(1)} \triangleleft \overline{G}^{(0)}$
is normal.
\end{lem}
\begin{cor}
\label{cor:normalg}
Let $G$ be a solvable subgroup of $\diffh{}{n}$.
Denote $H = \overline{G}_{u}^{(0)}$.
Then   $\overline{H}^{(j)}$ is the closure
in the Krull topology of the $j$-derived group
of   $H$ for any $j \in {\mathbb N}$. Moreover, the series
\[  \hdots \triangleleft \overline{H}^{(m)}
\triangleleft \hdots \triangleleft \overline{H}^{(1)} \triangleleft H=\overline{H}^{(0)} \]
is normal. The group $\overline{H}^{(j)}$ is a normal subgroup of $\overline{G}^{(0)}$
for any $j \geq 0$.
\end{cor}
Given a subgroup $G$ of $\diffh{}{n}$ we construct its pro-algebraic closure
$\overline{G}^{(0)}$. The pro-algebraic nature extends to the closed derived groups.
\begin{pro}
\label{pro:proa}
Let $G$ be a subgroup of $\diffh{}{n}$. Then
$\overline{G}^{(j)}$ is a pro-algebraic group for any $j \in {\mathbb N} \cup \{0\}$.
More precisely $\{ \varphi_{k} : \varphi \in \overline{G}^{(j)} \}$ is the
algebraic matrix group  $G_{k}^{(j)}$ for all
$j \in {\mathbb N} \cup \{0\}$ and $k \in {\mathbb N}$ and we have
$\overline{G}^{(j)} = \varprojlim G_{k}^{(j)}$
for any $j \in {\mathbb N} \cup \{0\}$.
\end{pro}
\begin{proof}
The derived group of a linear algebraic group is algebraic (cf.  \cite{Borel}[2.3, p. 58]).
As a consequence $G_{k}^{(j)}$ is algebraic for all
$j \in {\mathbb N} \cup \{0\}$ and $k \in {\mathbb N}$.

We define $\tilde{G}^{(j)} = \{ \varphi  \in \diffh{}{n} : \varphi_{k} \in G_{k}^{(j)}  \ \forall k \in {\mathbb N} \}$.
Since $\pi_{l,k} (G_{l})= G_{k}$, we obtain $\pi_{l,k} (G_{l}^{(j)})= G_{k}^{(j)}$
for all $l \geq k \geq 1$ and $j \geq 0$.
The group $\tilde{G}^{(j)}$ is pro-algebraic for any $j \geq 0$ by Lemma \ref{lem:cons}.

The remainder of the proof is devoted to show $\overline{G}^{(j)} = \tilde{G}^{(j)}$
for any $j \geq 0$.
It suffices to prove the result for $j=1$.
The inclusion $\overline{G}^{(1)} \subset \tilde{G}^{(1)}$ is clear.

Let $\varphi \in \tilde{G}^{(1)}$. Fix $k \in {\mathbb N}$.
Then $\varphi_{k}$ is a product of commutators of elements of $G_{k}$.
Since $\pi_{j,k}: G_{j} \to G_{k}$ is surjective for $j \geq k$, we obtain
that there exists $\eta(k) \in (\overline{G}^{(0)})^{(1)}$ such that $(\eta(k))_{k}=\varphi_{k}$.
Therefore $\varphi$ is the limit in the Krull topology of the sequence
$(\eta(k))_{k \geq 1}$. We are done since the Krull closure of $(\overline{G}^{(0)})^{(1)}$ is equal
to $\overline{G}^{(1)}$ by definition.
\end{proof}
\begin{rem}
The previous results justify the definition of $\overline{G}^{(j)}$.
On the one hand $\overline{G}^{(j)}= \{Id\}$ is equivalent to
${G}^{(j)}=\{Id\}$ by Corollary \ref{cor:normalg}
and Lemma \ref{lem:solaux}.
On the other hand the group $\overline{G}^{(j)}$ is more
compatible with the pro-algebraic nature of $\overline{G}^{(0)}$
than $(\overline{G}^{(0)})^{(j)}$ by Proposition \ref{pro:proa}.
\end{rem}
Let us introduce the closed derived series of a Lie algebra.
\begin{defi}
Let ${\mathfrak g}$ be a Lie subalgebra of $\hat{\mathfrak X} \cn{n}$.
We denote by $\overline{\mathfrak g}^{(0)}$ the closure of ${\mathfrak g}$
in the Krull topology.
We define the $j$-closed derived Lie algebra $\overline{\mathfrak g}^{(j)}$
of ${\mathfrak g}$ as the closure in the Krull topology of
$[\overline{\mathfrak g}^{(j-1)}, \overline{\mathfrak g}^{(j-1)}]$ for any
$j \in {\mathbb N}$.
\end{defi}
\begin{defi}
Let $G$ be a solvable subgroup of $\diffh{}{n}$.
We denote $\overline{\mathfrak g}_{N}^{(j)} = \overline{({\mathfrak g}_{N})}^{(j)}$.
\end{defi}
The next lemma is the analogue of Lemma \ref{lem:normalg} for Lie algebras.
\begin{lem}
\label{lem:normal}
Let $G \subset \diffh{}{n}$ be a solvable group.
Then $\overline{\mathfrak g}^{(j)}$ (resp. $\overline{\mathfrak g}_{N}^{(j)}$)
is the closure in the Krull topology of ${\mathfrak g}^{(j)}$
(resp. $({\mathfrak g}_{N})^{(j)}$) for any $j \in {\mathbb N}$.
Moreover we have
$\varphi_{*} \overline{\mathfrak g}^{(j)} = \overline{\mathfrak g}^{(j)}$
and $\varphi_{*} \overline{\mathfrak g}_{N}^{(j)} = \overline{\mathfrak g}_{N}^{(j)}$
for all $\varphi \in \overline{G}^{(0)}$ and $j \in {\mathbb N}$. The series
\[ \hdots \triangleleft \overline{\mathfrak g}^{(m)}
\triangleleft \hdots \triangleleft \overline{\mathfrak g}^{(1)}
\triangleleft \overline{\mathfrak g}^{(0)} = {\mathfrak g}, \ \
\hdots \triangleleft \overline{\mathfrak g}_{N}^{(m)}
\triangleleft \hdots \triangleleft \overline{\mathfrak g}_{N}^{(1)}
\triangleleft \overline{\mathfrak g}_{N}^{(0)} = {\mathfrak g}_{N} \]
are normal. The Lie algebra $\overline{\mathfrak g}_{N}^{(j)}$ is an ideal of
${\mathfrak g}$ for any $j \geq 0$.
\end{lem}
The next proposition establishes that the derived length of a unipotent subgroup of
$\diffh{}{n}$ and its Lie algebra coincide.
\begin{pro}
\label{pro:lieder} \cite{JR:arxivdl}[Proposition 3]
Let $G \subset \diffh{}{n}$ be a solvable group.
Denote $H = \overline{G}_{u}^{(0)}$.
Then $\overline{\mathfrak g}_{N}^{(j)}$ is the Lie algebra of
$\overline{H}^{(j)}$ for any $j \in {\mathbb N}$. In particular we obtain
${\rm exp}(\overline{\mathfrak g}_{N}^{(j)}) = \overline{H}^{(j)}$.
\end{pro}
Given a connected Lie group $G$ with Lie algebra ${\mathfrak g}$,
the Lie algebra of $G^{(1)}$ is the derived Lie algebra ${\mathfrak g}^{(1)}$.
Proposition \ref{pro:lieder} is an analogue of such result adapted to the context
of connected pro-algebraic groups.
\subsection{Classical results}
Let us introduce classic results, namely theorems by Lie and Kolchin
on triangularizable Lie algebras and groups and
the Baker-Campbell-Hausdorff formula.
\begin{teo}
\label{teo:Lie}
(Lie, cf. \cite{Serre.Lie}[chapter V, section 5, p. 36])
Let ${\mathfrak g}$ be a solvable Lie algebra over an algebraically closed field
$K$ of characteristic $0$. Let $\varrho$ be a linear representation of
${\mathfrak g}$ with representation space $V$. Then up to change of base in $V$
the Lie algebra $\varrho({\mathfrak g})$ consists of upper triangular matrices.
\end{teo}
\begin{teo}
\label{teo:Lie-Kolchin}
(Lie-Kolchin, cf. \cite{Humphreys}[section 17.6, p. 113])
Let $G$ be a solvable connected subgroup of $\mathrm{GL}(n, F)$ where $F$ is an algebraically
closed field. Then up to a change of base $G$ is a group of upper triangular matrices.
\end{teo}
\begin{teo}
\label{teo:Kolchin}
(Kolchin, cf. \cite{Serre.Lie}[chapter V, p. 35]).
Let $V$ a finite dimensional vector space over a field $K$.
Let $G$ be a subgroup of $\mathrm{GL} (V)$ such that each element $g \in G$ is
unipotent. Then up to a change of base $G$ is a group of upper triangular
matrices.
\end{teo}
\begin{rem}
\label{rem:unimnil}
A unipotent matrix group is nilpotent by Kolchin's theorem.
\end{rem}
\begin{rem}
\label{rem:unimsc}
Let $G$ be a unipotent linear algebraic group over ${\mathbb C}$ with
Lie algebra ${\mathfrak g}$.
The exponential mapping
${\rm exp}: {\mathfrak g} \to G$ is a diffeomorphism. Thus $G$ is
simply connected.
\end{rem}
%
%
\begin{rem} (Baker-Campbell-Hausdorff formula, cf. \cite{Corwin-Greenleaf}[ chapter 1]).
Let $G$ be a simply connected nilpotent Lie group with
Lie algebra ${\mathfrak g}$. The mapping
${\rm exp}: {\mathfrak g} \to G$ is an analytic diffeomorphism.
Let $X, Y \in {\mathfrak g}$. We have that $\log({\rm exp} (X){\rm exp} (Y))$ is equal to
\begin{equation}
\label{equ:BCH}
 \sum_{n>0}\frac {(-1)^{n-1}}{n} \sum_{ \begin{smallmatrix} {r_i + s_i > 0} \\ {1\le i \le n} \end{smallmatrix}}
\frac{(\sum_{i=1}^n (r_i+s_i))^{-1}}{r_1!s_1!\cdots r_n!s_n!}  [ X^{r_1} Y^{s_1} X^{r_2} Y^{s_2} \ldots X^{r_n} Y^{s_n} ]
\end{equation}
where
\[ [ X^{r_1} Y^{s_1} \ldots X^{r_n} Y^{s_n} ] =
[ \underbrace{X,\ldots[X}_{r_1} , \,\ldots\,
[ \underbrace{X,[X,\ldots[X}_{r_n} ,[ \underbrace{Y,[Y,\ldots Y}_{s_n} ]]\ldots]]. \]
The previous formula is due to Dynkin. Remarks \ref{rem:unimnil} and \ref{rem:unimsc}
imply that we can apply this formula in a unipotent subgroup $G$ of $\diffh{}{n}$.
Indeed $G_{k}$ is a unipotent simply connected nilpotent matrix group for any
$k \in {\mathbb N}$.
\end{rem}
\begin{rem}
The study of the action of $\overline{G}^{(0)}$ on the Lie algebra ${\mathfrak g}_{N}$ is
crucial to study the derived length of solvable subgroups of $\diff{}{n}$.
The possibility of using Baker-Campbell-Hausdorff formula makes ${\mathfrak g}_{N}$
easier to handle than ${\mathfrak g}$.
\end{rem}
\section{Algebraic structure of unipotent groups}
\label{sec:alg}
We have $\ell (G) \leq \ell (G/G_{u}) + \ell (G_{u})$ for any
solvable subgroup $G$ of $\diffh{}{n}$.
The group $G/G_{u}$ is isomorphic to $j^{1} G / (j^{1} G)_{u}$.
It is natural to study the algebraic properties of $L/L_{u}$ for a
solvable subgroup $L$ of $\mathrm{GL}(n,{\mathbb C})$ and of
unipotent subgroups of $\diffh{}{n}$. The former topic is treated in section \ref{sec:linear}
whereas the latter one is the subject of this section.
Of course calculating sharp bounds for the derived length in these two cases
does not suffice to find a sharp bound for solvable subgroups of $\diffh{}{n}$
since these problems are not independent in our context. In fact
$G^{(\sigma)}$ can be much smaller than $G_{u}$ for
$\sigma = \ell (G/G_{u})$. Sections \ref{sec:2} and \ref{sec:35} are devoted
to understand how these processes interact with each other.
\begin{defi}
We define $\hat{K}_{n}$ the field of fractions of the ring
of formal power series $\hat{\mathcal O}_{n}$.
\end{defi}
\begin{defi}
Let $G \subset \diffh{}{n}$ be a solvable group.
We define $\kappa (p)$
as the dimension of the ${\hat{K}_{n}}$-vector space
$\overline{\mathfrak g}_{N}^{(p)}  \otimes_{\mathbb C} \hat{K}_{n}$
for any $p \geq 0$.
It is obvious that $\kappa (p+1) \leq \kappa (p) \leq n$ for any $p \geq 0$.
\end{defi}
\begin{rem}
In \cite{JR:arxivdl} we defined $\kappa (p)$ as the dimension of
$\overline{\mathfrak g}^{(p)}  \otimes_{\mathbb C} \hat{K}_{n}$.
Here we focus on the Lie algebra of the unipotent part of $G$.
\end{rem}
\begin{defi}
Let ${\mathfrak g}$ be a Lie subalgebra of $\hat{\mathfrak X} \cn{n}$.
We define the field
\[ {\mathcal M}({\mathfrak g}) = \{ g \in \hat{K}_{n} : X(g) =0  \
\forall X \in {\mathfrak g} \} \]
of ``formal meromorphic" first integrals of ${\mathfrak g}$.
Let $G \subset \diffh{}{n}$ be a solvable group. We define
${\mathcal M}_{p} =  {\mathcal M}(\overline{\mathfrak g}_{N}^{(p)})$
for any $p \geq 0$.
\end{defi}
Given a solvable subgroup $G$ of $\diffh{}{n}$ the Lie algebra
$\overline{\mathfrak g}_{N}^{(p)}$ is invariant by the action of elements of
$\overline{G}^{(0)}$ for $p \geq 0$. As a consequence
$g \circ \varphi  \in {\mathcal M}_{p}$ for all $\varphi \in \overline{G}^{(0)}$ and
$g \in {\mathcal M}_{p}$. The actions defined by $G$ in $\overline{\mathfrak g}_{N}^{(p)}$
and ${\mathcal M}_{p}$ are crucial to understand the algebraic properties of the derived
series of $G$.

The next lemmas are key results in order to provide the classification of
subgroups of $\diffh{}{n}$ whose first jet is a connected matrix group
\cite{JR:arxivdl}[Theorem 6].
\begin{lem}
\label{lem:aux} \cite{JR:arxivdl}
Let $G \subset \diffh{}{n}$ be a solvable group.
Suppose $\kappa(p+1) < \kappa (p)$.
Denote $q=\kappa (p+1)$ and $m=\kappa(p)-\kappa(p+1)$.
Let $\{ Y_{1}, \hdots, Y_{q} \}     \subset \overline{\mathfrak g}_{N}^{(p+1)}$
be a base of the $\hat{K}_{n}$-vector space
$\overline{\mathfrak g}_{N}^{(p+1)}  \otimes_{\mathbb C} \hat{K}_{n}$.
Let $\{ X_{1}, \hdots, X_{m} \} \subset \overline{\mathfrak g}_{N}^{(p)}$
be a set such that
$\{ Y_{1}, \hdots, Y_{q} , X_{1}, \hdots, X_{m} \}$
is a base of
$\overline{\mathfrak g}_{N}^{(p)}  \otimes_{\mathbb C} \hat{K}_{n}$.
Consider
$Z = \sum_{j=1}^{q} b_{j} Y_{j} + \sum_{k=1}^{m} a_{k} X_{k} \in
\overline{\mathfrak g}_{N}^{(p)}$. Then
$a_{1}, \hdots , a_{m}$ belong to ${\mathcal M}_{p}$.
\end{lem}
\begin{lem}
\label{lem:aux2} \cite{JR:arxivdl}
Let $G \subset \diffh{}{n}$ be a solvable group.
Suppose $\kappa(p+1) < \kappa (p)$ and
$\kappa(r) = \kappa (p)$ for some $r < p$.
Consider the notations in Lemma \ref{lem:aux}.
Let  $Z = \sum_{j=1}^{q} b_{j} Y_{j} + \sum_{k=1}^{m} a_{k} X_{k}$
be an element of $\overline{\mathfrak g}_{N}^{(r)}$. Then
$a_{k}$ belongs to ${\mathcal M}_{p+1}$ for any $1 \leq k \leq m$
and
$X_{j}(a_{k}) \in {\mathcal M}_{p}$ for all $1 \leq j \leq m$ and $1 \leq k \leq m$
\end{lem}
The Propositions \ref{pro:maxd} and \ref{pro:intd} are useful to understand the nature
of ${\mathfrak g}_{N}$.
\begin{pro}
\label{pro:maxd} \cite{JR:arxivdl}[Proposition 10]
Let $G \subset \diffh{}{n}$ be a solvable group.
Suppose $\kappa(p)>0$ and ${\mathcal M}_{p} = {\mathbb C}$
for some $p \geq 0$.
Then we have $\kappa(p+1) < \kappa(p)$.
\end{pro}
\begin{cor}
\label{cor:maxd}
Let $G \subset \diffh{}{n}$ be a solvable group
such that $\kappa (0)=n$. Then we have $\kappa(1) < n$.
\end{cor}
\begin{pro}
\label{pro:intd} \cite{JR:arxivdl}[Proposition 8]
Let $G \subset \diffh{}{n}$ be a solvable group.
Suppose $\kappa(p)>0$
for some $p \geq 0$.
Then we have $\kappa(p+2) < \kappa(p)$.
\end{pro}
\begin{rem}
\label{rem:dim1}
Let $G \subset \diffh{}{n}$ be a unipotent solvable group with $\kappa (0) \leq p$.
Then we obtain $\ell (G) \leq 2p$ by Propositions \ref{pro:lieder} and
\ref{pro:intd}.
\end{rem}
Corollary \ref{cor:maxd} and Propositions \ref{pro:intd} and \ref{pro:lieder}
imply the inequality in next theorem.
\begin{teo}
\label{teo:uni} \cite{JR:arxivdl}[Theorem 4]
We have $\ell (G) \leq 2n-1$ for any
unipotent solvable subgroup $G$ of $\diffh{}{n}$.
Moreover the bound is sharp, i.e. there exists a unipotent solvable
subgroup $H$ of $\mathrm{Diff}({\mathbb C}^{n},0)$ such that
$\ell (H)=2n-1$.
\end{teo}
\section{Linear groups}
\label{sec:linear}
Fix a field $F$ and $n \in {\mathbb N}$.
As explained in section \ref{sec:alg} we are interested in
calculating
$\max \{ \ell (G/G_{u}) : G < \mathrm{GL}(n,F) \ \mathrm{and} \ \ell (G) < \infty \}$.
Let us explain how this problem is related to the solvability properties
of completely reducible groups.

Consider a subgroup $G$ of $\mathrm{GL}(n,F)$.
We introduce a construction that is used by Dixon in \cite{Dixon-solvable}.
Consider a sequence
\[ \{0\} = V_{0} \subset V_{1} \subset \hdots \subset V_{k} = F^{n} \]
of $G$-invariant linear subspaces such that
$V_{j} \neq V_{j+1}$ and the action of $G$ on $V_{j+1}/V_{j}$ is irreducible for
any $0 \leq j < k$. It is always possible up to introduce more subspaces in the 
sequence.
We denote by $A_{|V_{j+1}/V_{j}}$ the action of $A$ on $V_{j+1}/V_{j}$.
We define the map
\[ \begin{array}{ccccc}
\varpi & : & G & \to & \mathrm{GL} (V_{1}/V_{0}) \times \hdots \times \mathrm{GL} (V_{k}/V_{k-1}) \\
      &   & A & \mapsto & \left( A_{|V_{1}/V_{0}}, \hdots , A_{|V_{k}/V_{k-1}} \right). \\
\end{array}
\]
By construction the group $\varpi (G)$ is a completely reducible subgroup of
$\mathrm{GL}(n,F)$. The kernel of the morphism $\varpi$ is contained in the
set of
unipotent elements; this can be very easily proved by considering a base of
$F^{n}$ that contains a base of $V_{j}$ for any $0 \leq j \leq k$.
We can be more precise in the solvable case.
\begin{pro}
\label{pro:cr}
Let $G$ be a solvable subgroup of $\mathrm{GL}(n, F)$ where $F$
is a field of characteristic $0$. Then
we have $\ker (\varpi) = G_{u}$.
In particular $G/G_{u}$ is isomorphic to
a completely reducible subgroup of $\mathrm{GL}(n,F)$.
Moreover if $G$ is completely reducible then $G_{u}=\{Id\}$.
\end{pro}
\begin{proof}
We already proved $\ker (\varpi) \subset G_{u}$. Let us show $G_{u} \subset \ker (\varpi)$.

We consider $G$ as a subgroup of $\mathrm{GL}(n,\overline{F})$ where $\overline{F}$
is the algebraic closure of $F$.
Let $\overline{G}$ be the algebraic closure of $G$.
The set of unipotent elements $\overline{G}_{u}$ of $\overline{G}$
is contained in the connected component of the identity $\overline{G}_{0}$
(cf. \cite{Humphreys}[Lemma C, p. 96]).
We apply Lie-Kolchin's theorem to obtain that up to a change of base
$\overline{G}_{0}$ is a group of upper triangular matrices.
Since $\overline{G}_{u}$ is the subgroup of $\overline{G}_{0}$
of matrices whose diagonal elements are all equal to $1$,
$G_{u}$ is a normal subgroup of $G$.

The group $\varpi (G_{u})$ is a normal subgroup of $\varpi (G)$.
A normal subgroup of a completely reducible linear group is completely
reducible by Clifford's theorem (cf. \cite{Wehrfritz}[Corollary 1.8, p. 5]).
Then $\varpi (G_{u})$ is a completely reducible unipotent subgroup of
$\mathrm{GL}(n,F)$.
A non-trivial unipotent subgroup of $\mathrm{GL} (n, F)$ is never
completely reducible by Kolchin's theorem.
We deduce $\varpi (G_{u}) = \{Id\}$ (and even $G_{u}=\{ Id \}$ if
$G$ is completely reducible).
Hence $G_{u}$ is contained in $\ker (\varpi)$.
\end{proof}
In \cite{Newman} Newman calculates the maximum derived length $\sigma (n)$ for
completely reducible solvable subgroups of $\mathrm{GL}(n,F)$ where $F$ is any field.
Indeed we have $\sigma (1)=1$, $\sigma (2)=4$, $\sigma (3)=5$, $\sigma (4)=5$ and
$\sigma (5)=5$. This function does not coincide with the Newman function for
$n=3,4,5$ since $\rho (3) =5$, $\rho (4)= 6$ and $\rho (5)=7$.
\begin{defi}
\label{def:psi}
We define
\[ \psi (n) = \max \{ \ell (G) : G \ \mathrm{is} \ \mathrm{a} \ \mathrm{solvable} \
\mathrm{subgroup} \ \mathrm{of} \ \diffh{}{n} \} \]
for $n  \in {\mathbb N}$.
\end{defi}
\begin{pro}
\label{pro:ibound}
Let $G$ be a solvable subgroup of $\diffh{}{n}$.
Then $\ell (G) \leq \sigma (n) + \ell (G_{u})$. In particular we obtain
$\psi (n) \leq \sigma (n) + 2n -1$.
\end{pro}
\begin{proof}
The natural map $G/G_{u} \to j^{1} G/ (j^{1} G)_{u}$ is an isomorphism.
Therefore $G/G_{u}$ is isomorphic to a completely reducible subgroup of
$\mathrm{GL}(n,{\mathbb C})$ by Proposition \ref{pro:cr}.
We deduce $G^{(\sigma(n))} \subset G_{u}$ by definition of $\sigma$.
Thus $\ell (G)$ is less or equal than $\sigma (n) + \ell (G_{u})$.
Since $\ell (G_{u}) \leq 2n-1$ by Theorem \ref{teo:uni}, we obtain
$\ell (G) \leq \sigma (n) + 2n -1$.
\end{proof}
The function $\sigma (n) + 2n -1$ is equal to $2$, $7$, $10$, $12$ and $14$
for $n=1,2,3,4,5$ respectively. Let us remind the reader that we want to
show $\psi (2)=5$, $\psi (3)=7$, $\psi (4)=9$ and $\psi (5)=11$.

In the proof of the Main Theorem
we need to understand
the group $G^{(\sigma(n)-1)}/(G^{(\sigma(n-1))})_{u}$ for a solvable
subgroup $G$ of $\mathrm{GL}(n, {\mathbb C})$.
\begin{lem}
\label{lem:fag}
Let $G$ be a completely reducible solvable subgroup of $\mathrm{GL}(n,F)$
where $F$ is a field of characteristic $0$ and $n>1$.
Then $G^{(\sigma (n)-1)}$ is a finite abelian group.
\end{lem}
\begin{proof}
We can suppose that $F$ is algebraically closed.
Indeed $G_{u}$ is trivial by Proposition \ref{pro:cr}
and $G = G/G_{u}$ is isomorphic to a completely reducible
subgroup of $\mathrm{GL}(n,\overline{F})$ by Proposition \ref{pro:cr}.
More generally given fields $F \subset K$
any completely reducible subgroup of $\mathrm{GL}(n,F)$
is isomorphic to a completely reducible subgroup of $\mathrm{GL}(n,K)$
(cf. \cite{Dixon-solvable}[remark after Lemma 2]).

Let us show the lemma by induction on $n$.
Since $\sigma$ is an increasing function, it suffices to prove the result for
an irreducible group $G$.

We denote $\sigma = \sigma (n)$.
The group $G^{(\sigma -1)}$ is commutative.
Let us remark that a solvable subgroup $H$ of $\mathrm{GL}(n,F)$
is completely reducible if and only if all its elements are diagonalizable
(cf. \cite{Wehrfritz}[Theorem 7.6, p. 94]).
In particular every element of $G$ is diagonalizable.
We denote by $W_{\lambda}(A)= {\rm Ker}(A-\lambda Id)$ the
eigenspace of $\lambda$. If $A$ and $B$ commute then we have
$B(W_{\lambda}(A))=W_{\lambda}(A)$.
The vector space $F^{n}$ is of the form
$W_{1} \bigoplus \hdots \bigoplus W_{k}$ where every
$W_{j}$ is defined as follows: for any $A \in  G^{(\sigma -1)}$ we choose
an eigenvalue $\lambda(A,j)$ of $A$; we define
$W_{j}=\cap_{A \in G^{(\sigma - 1)}} W_{\lambda(A,j)}(A)$. Of course we remove
all the trivial spaces $W_{j}=\{0\}$.
The decomposition is preserved by $G$ since $G^{(\sigma -1)}$ is normal in $G$.
Moreover since $G$ is irreducible, given $1 \leq j \leq k$ there
exists $A_{j} \in G$ such that $A_{j} (W_{1})=W_{j}$.
In particular there exists $p \in  {\mathbb N}$ such that $\dim (W_{j})=p$
for any $1 \leq j \leq k$.

We have $A (W_{j}) = W_{j}$ for all $A \in G^{(\tau (k))}$ and $1 \leq j \leq n$
where $\tau (k)$ is the sharpest upper bound for the derived length of solvable
subgroups of the group $S_{k}$ of permutations of $k$ elements.
The function $\tau$ is extended to the rational positive numbers in \cite{Newman}.
The properties of the function $\tau$ will be used below.
The subrepresentation $G^{(\tau (k))}_{|W_{j}}$ obtained by restricting
$G^{(\tau (k))}$ to $W_{j}$ is also completely reducible
since all its elements are diagonalizable.

Suppose $p=1$. We have $n=k$. We obtain $G^{(\tau (k)+1)}_{|W_{j}}=\{Id\}$
for any $1 \leq j \leq k$.
Hence it suffices to prove $\tau (n) + 1 \leq \sigma (n) -1$ for any $n \geq 2$.
We have $\sigma (n) = 7 + \tau (n/8)$ or $\sigma (n) = 8 + \tau (n/8)$ \cite{Newman}[Theorem C].
It suffices to prove $\tau (n) \leq 5+ \tau (n/8)$ for any $n \in {\mathbb N}$.
We obtain $5+ \tau (n/8) = \tau (9n/8)$ \cite{Newman}[Lemma 1(a)(5)].
Now $\tau (n) \leq \tau (9n/8)$ is a consequence of the increasing nature of $\tau$.
We suppose $p>1$ from now on.

Suppose $k>1$. We obtain that
$G^{(\tau (k)+\sigma (p) -1)}_{|W_{j}}$ is a finite abelian group
for any $1 \leq j \leq k$ by induction hypothesis.
Therefore $G^{(\tau (k)+\sigma (p) -1)}$ is a finite abelian group.
The property
$\tau (k) + \sigma (p) \leq  \sigma (pk)$ \cite{Newman}[Lemma 1(c)]
implies that $G^{(\sigma -1)}$ is a finite abelian group.

Suppose $k=1$. We obtain that every element $A$ of $G^{(\sigma -1)}$
is of the form $\lambda_{A} Id$ by construction of
$W_{1} \bigoplus \hdots \bigoplus W_{k}$.
If a commutator $[A,B]$ is of the form $\lambda Id$ then
$ABA^{-1} = \lambda B$. The eigenvalues of $B$ and $\lambda B$ coincide;
thus $\lambda^{n} =1$. Since $\sigma \geq 4$ we deduce that
$G^{(\sigma -1)}$ is a subgroup of $\{ \lambda Id: \lambda^{n}=1 \}$.

The induction hypothesis is only used to deal with the case $p>1$, $k>1$.
Since then $pk \geq 4$ the result holds true for the initial case $n=2$.
\end{proof}
\begin{cor}
\label{cor:lin2}
Consider a solvable subgroup
$G$ of $\mathrm{GL} (2,F)$ where $F$ is a field of characteristic $0$. Then we have
$G^{(3)} = \{ Id \}$ or $G^{(3)} = \{ Id, - Id \}$.
\end{cor}
\begin{proof}
We can suppose that $F$ is algebraically closed. 
Moreover we can suppose that $G$ is irreducible. Otherwise $G$ is conjugated to a group
of upper triangular matrices and $G^{(2)}=\{Id\}$.

Consider the cases in the proof of Lemma \ref{lem:fag}.
If $p=1$ and $k=2$ then we obtain again $G^{(2)}=\{Id\}$.
If $p=2$ and $k=1$ then $G^{(3)}$ is a subgroup of $\{-Id, Id\}$.
\end{proof}
\begin{cor}
\label{cor:fag}
Let $G$ be a solvable subgroup of $\diffh{}{n}$ with $n>1$.
Then $\overline{G}^{(\sigma (n)-1)}/(\overline{G}^{(\sigma (n)-1)})_{u}$
is a finite abelian group.
\end{cor}
\begin{proof}
We denote $L = \overline{G}^{(0)}$ and $T=L^{(\sigma (n)-1)}$. Notice that
$\overline{G}^{(\sigma (n)-1)}$ is the Krull closure $\overline{T}$ of $T$ by
Lemma \ref{lem:normalg}.

The natural map
$L/L_{u} \to j^{1} L/ (j^{1}  L)_{u}$
is an isomorphism of groups.
We deduce that $T/T_{u}$ is a finite abelian group
by Proposition \ref{pro:cr} and Lemma \ref{lem:fag}.
The groups $T/T_{u}$ and $j^{1}  T/ (j^{1} T)_{u}$ are isomorphic.
We also have that $\overline{T}/(\overline{T})_{u}$ is isomorphic to
$j^{1} T/ (j^{1} T)_{u}$ by construction of $\overline{T}$.
Therefore $\overline{G}^{(\sigma (n)-1)}/(\overline{G}^{(\sigma (n)-1)})_{u}$
is a finite abelian group.
\end{proof}
%
\section{Codimension one foliations}
%
%
Our point of view is based on studying the action of $G$
on Lie algebras and fields of first integrals.
Proposition \ref{pro:cod1} implies that if the first integrals
define a codimension $1$ foliation
then the action of the group on first integrals
is of derived length at most $3$.
\begin{pro}
\label{pro:cod1aux}
Let $X_{1},\hdots, X_{q} \in \hat{\mathfrak X} \cn{n}$.
Denote
\[ {\mathcal M} = \{ g \in \hat{K}_{n} : X_{j}(g)=0 \ \forall j \in \{1, \hdots, q\}\} . \]
Suppose ${\mathcal M}$ defines a codimension $1$ formal foliation, i.e.
there exists $f_{1} \in {\mathcal M} \setminus {\mathbb C}$ and
$df \wedge dg \equiv 0$ for all $f,g \in {\mathcal M}$. Then
\begin{itemize}
\item If ${\mathcal M} \cap \hat{\mathcal O}_{n} \neq {\mathbb C}$ then
there exists $f_{0} \in \hat{\mathfrak m}$ such that
${\mathcal M} = {\mathbb C}[[f_{0}]][f_{0}^{-1}]$ and
${\mathcal M} \cap \hat{\mathcal O}_{n}= {\mathbb C}[[f_{0}]]$.
\item If ${\mathcal M} \cap \hat{\mathcal O}_{n} = {\mathbb C}$ then
there exists $f_{0} \in {\mathcal M}$ such that
${\mathcal M} = {\mathbb C} (f_{0})$.
\end{itemize}
\end{pro}
\begin{pro}
\label{pro:cod1}
Consider the hypotheses in Proposition \ref{pro:cod1aux}.
Suppose that
$G$ is a solvable pro-algebraic subgroup of $\diffh{}{n}$ such that 
${\mathcal M} \circ \varphi = {\mathcal M}$ for any $\varphi \in G$.
Then $f \circ \varphi = f$ for all $f \in {\mathcal M}$ and
$\varphi \in \overline{G}^{(3)}$.
Moreover we obtain
$f \circ \varphi = f$ for all $f \in {\mathcal M}$ and
$\varphi \in \overline{G}^{(2)}$ if
${\mathcal M} \cap \hat{\mathcal O}_{n}$ contains non-constant elements.
\end{pro}
\begin{proof}[Proof of Proposition \ref{pro:cod1aux}]
Denote $\tilde{\mathcal M}={\mathcal M} \cap \hat{\mathcal O}_{n}$ and
$\tilde{\mathcal M}_{0}={\mathcal M} \cap \hat{\mathfrak m}$.
Suppose  that $df_{1}=0$ defines a foliation, i.e. there exists a
non-vanishing $\omega \in \Omega^{1} \cn{n}$
such that $\omega \wedge df_{1} =0$.
If ${\mathcal M}$ contains a non-constant element of $\hat{\mathcal O}_{n}$
then ${\mathcal M} \cap \hat{\mathcal O}_{n} = {\mathbb C}[[z]] \circ f_{0}$
for some $f_{0} \in \tilde{\mathcal M}_{0}$
\cite{MaMo:Aen}. Moreover $f_{0}$ can be any element of
$\tilde{\mathcal M}_{0}$ that is reduced, i.e.
there are no $g \in  \hat{\mathcal O}_{n}$ and $m \geq 2$
such that $f_{0}=g^{m}$. Moreover any element $g$ of ${\mathcal M}$
satisfies either that either $g$ or $1/g$ belongs to
${\mathcal M} \cap \hat{\mathcal O}_{n}$. The other case
corresponds to $\tilde{\mathcal M}={\mathbb C}$. Then
there exists a non-constant $f_{0} \in {\mathcal M}$ such that
${\mathcal M} = {\mathbb C} (f_{0})$ \cite{Ce-Ma:Ast}.
The following paragraphs are intended to generalize the Mattei-Moussu
and Cerveau-Matei results to the formal setting.

Suppose there exists a pure meromorphic first integral, i.e.
an element $f \in {\mathcal M}$ of the form $a/b$ where
$a$ and $b$ are coprime elements of $\hat{\mathfrak m}$.
We reduce the general case to the case $n=2$ by restricting the
elements of ${\mathcal M}$ to a generic linear subspace $L$ of dimension $2$.

Notice that
even if $df=0$ does not define a
foliation we can desingularize $df=0$.
Indeed it has the same desingularization process than some foliation
since equidesingularization of formal integrable $1$-forms (in dimension $2$)
is a finite determination property.

We claim that there exists a sequence
$\pi_{1}, \hdots, \pi_{p}$ of blow-ups, satisfying that
$\pi_{1}$ is the blow-up of the origin, $\pi_{j+1}$ is the blow-up
of a point in $(\pi_{1} \circ \hdots \circ \pi_{j})^{-1}(0)$ for any $1 \leq j < p$
and there exists an irreducible component $D$ of
$(\pi_{1} \circ \hdots \circ \pi_{p})^{-1}(0)$ such that
$(f \circ \pi_{1} \circ \hdots \circ \pi_{p})_{|D}$ is not constant.
Otherwise the desingularization of $df=0$ implies that are finitely many
formal invariant curves, contradicting that $a - \lambda b=0$
is a formal invariant curve ($(a - \lambda b) | d(a- \lambda b) \wedge df$)
for any $\lambda \in {\mathbb C}$.

We can identify $D$ with ${\bf CP}^{1}$. Let ${\mathbb C} (z)$ be the field of rational functions
defined in the Riemann sphere.
We define the map
\[
\begin{array}{ccccc}
\tau & : & {\mathcal M} & \to & {\mathbb C} (z) \\
     &   & g & \mapsto & (g \circ \pi_{1} \circ \hdots \circ \pi_{p})_{|D}.
\end{array}
\]
Since $D$ is a dicritic component, the kernel of $\tau$ is trivial.
We deduce that $\tau: {\mathcal M} \to \tau({\mathcal M})$ is an isomorphism
of fields. Since an element $g$ of $\tilde{\mathcal M}$
satisfies that $(g \circ \pi_{1} \circ \hdots \circ \pi_{p})_{|D}$ is constant,
we deduce $\tilde{\mathcal M}  = {\mathbb C}$.
A subfield of ${\mathbb C} (z)$, and in particular $\tau ({\mathcal M})$, is of the form
${\mathbb C}(\iota)$ for some $\iota \in {\mathbb C} (z)$ by L\"{u}roth's theorem
(cf. \cite{Waerden}[p. 198]).
Hence there exists $f_{0} \in {\mathcal M} \setminus {\mathbb C}$ such that
${\mathcal M} = {\mathbb C} (f_{0})$.

Suppose   $\tilde{\mathcal M} \neq {\mathbb C}$.
Since $\tilde{\mathcal M} \neq {\mathbb C}$ implies that there are no
pure meromorphic first integrals, we obtain that
any $f \in {\mathcal M}$ satisfies either
$f \in \tilde{\mathcal M}$ or $1/f \in \tilde{\mathcal M}$.

Given $g \in \hat{\mathcal O}_{n}$ we define $\nu (g)$ as the vanishing order of $g$
at the origin, i.e. we have $g \in \hat{\mathfrak m}^{\nu (g)} \setminus \hat{\mathfrak m}^{\nu(g) +1}$.
Let  $f_{0}$ be an element of $\tilde{\mathcal M}_{0} \setminus \{ 0 \}$ such that
$\nu (f_{0})$ is minimal.

We claim that $\tilde{\mathcal M}= {\mathbb C}[[f_{0}]]$. It suffices to show  $\tilde{\mathcal M} \subset {\mathbb C}[[f_{0}]]$.
Consider an element $g$ of $\tilde{\mathcal M}$.
There exists $c_{0} \in {\mathbb C}$ such that $g - c_{0} \in \tilde{\mathcal M}_{0}$.
We have a unique expression $\nu (g-c_{0}) = q \nu (f_{0}) + r$ where $q \in {\mathbb Z}_{\geq 0}$
and $0 \leq r < \nu (f_{0})$. The quotient $g_{q}:=(g-c_{0}) /f_{0}^{q}$ has vanishing order $r \geq 0$
and it is a formal power series (since it can not be the inverse of a vanishing power series).
Since $r < \nu (f_{0})$ we deduce $r=0$ and hence $g$ is of the form
$c_{0} + g_{q} f_{0}^{q}$ where $g_{q}$ belongs to $\tilde{\mathcal M} \setminus {\mathcal M}_{0}$.
We repeat the previous trick to obtain that $g$ is of the form
$\sum_{k=0}^{\infty} c_{k} f_{0}^{k}$ and as a consequence it belongs to ${\mathbb C}[[f_{0}]]$.

We obtain ${\mathcal M}= {\mathbb C}[[f_{0}]][f_{0}^{-1}]$
since either $f \in \tilde{\mathcal M}$ or $1/f \in \tilde{\mathcal M}$ for any $f \in {\mathcal M}$.
\end{proof}
\begin{proof}[Proof of Proposition \ref{pro:cod1}]
Given $\varphi \in G$ we obtain $f_{0} \circ \varphi = h_{\varphi} \circ f_{0}$
for some $h_{\varphi}$ since
${\mathcal M} \circ \varphi =  {\mathcal M}$.
Moreover $h_{\varphi}$ belongs to ${\mathbb C}(z)$
if there exists a pure meromorphic first integral.
Since $\tilde{\mathcal M}_{0} \circ \varphi =  \tilde{\mathcal M}_{0}$
the series $h_{\varphi}$ belongs to the maximal ideal of ${\mathbb C}[[z]]$ if
$\tilde{\mathcal M} \neq  {\mathbb C}$. It is clear that
$h_{\varphi} \circ h_{\varphi^{-1}} = Id =  h_{\varphi^{-1}}  \circ h_{\varphi}$
and $h_{\varphi \circ \eta} = h_{\varphi} \circ h_{\eta}$ for
$\varphi, \eta \in G$. We deduce that the map $\Lambda$ defined in $G$ by
$\Lambda (\varphi) = h_{\varphi}$ is of the form
\[
\begin{array}{ccc}
G & \stackrel{\Lambda} {\longrightarrow} & \mathrm{Aut} ({\bf CP}^{1}) \\
\varphi & \to & h_{\varphi}
\end{array}
\ \ \mathrm{or} \ \
\begin{array}{ccc}
G & \stackrel{\Lambda} {\longrightarrow} & \diffh{}{} \\
\varphi & \to & h_{\varphi}
\end{array}
\]
depending on wether or not $\tilde{\mathcal M}={\mathbb C}$.
The map $\Lambda$
is a homomorphism of groups.

In the former case
$\Lambda (G)$ is a solvable group of length at most $3$
by Corollary \ref{cor:lin2}. We
deduce that $f \circ \varphi = f$ for all
$f \in {\mathcal M}$ and $\varphi \in \overline{G}^{(3)}$.
In the latter case
the group $\Lambda (G)$ is metabelian
(cf. \cite{Ilya-Yako}[section $6B_{1}$]).
As a consequence $f \circ \varphi = f$ for all
$f \in {\mathcal M}$ and $\varphi \in \overline{G}^{(2)}$.
\end{proof}
\section{Soluble lengths of groups in dimension $2$}
\label{sec:2}
In this section we show that the derived length of a solvable subgroup of
$\diff{}{2}$ (or $\diffh{}{2}$) is less or equal than $5$.
The sharpness of the bound, that completes the proof of the Main Theorem
for $n=2$, is postponed to section \ref{sec:examples}.

The next lemma is of technical interest. It is a key ingredient in
the proof of the Main Theorem
for every dimension.
\begin{lem}
\label{lem:BCH}
Let $L$ be a solvable pro-algebraic subgroup of $\diffh{}{n}$.
Let ${\mathfrak l}_{N}$ be the Lie algebra associated to $L_{u}$
(cf. Definition \ref{def:liealg}).
Suppose there exists a set $B=\{ Y_{1}, \hdots, Y_{a}, X_{1}, \hdots, X_{b}\}$
of linear independent elements of
$\hat{\mathfrak X} \cn{n} \otimes_{\mathbb C} \hat{K}_{n}$
such that
\begin{enumerate}
\item ${\mathfrak l}_{N}$ is contained in the $\hat{K}_{n}$-vector space
generated by $B$.
\item $[Z,W] \in {\mathfrak l}'$ for all $Z,W \in B$ where
${\mathfrak l}'$ is the $\hat{K}_{n}$-vector space generated by
$\{ Y_{1}, \hdots, Y_{a} \}$.
\item Every element element $\alpha_{1} Y_{1} + \hdots + \alpha_{a} Y_{a} +
\beta_{1} X_{1} + \hdots + \beta_{b} X_{b}$ of ${\mathfrak l}_{N}$ satisfies
$Y_{k}(\beta_{j})=X_{l}(\beta_{j})=0$ and $\beta_{j} \circ \varphi =\beta_{j}$
for all $1 \leq j,l \leq b$, $1 \leq k \leq a$ and $\varphi \in L$.
\item $\varphi^{*} Y_{k} \in {\mathfrak l}'$ and $\varphi^{*} X_{j} - X_{j} \in {\mathfrak l}'$
for all $\varphi \in L$,  $1 \leq k \leq a$, $1 \leq j \leq b$.
\item $L/L_{u}$ is either cyclic or a finite abelian group.
\end{enumerate}
Then $\overline{L}^{(1)}$ is a unipotent group whose Lie algebra is contained in
${\mathfrak l}' \cap {\mathfrak l}_{N}$.
\end{lem}
\begin{proof}
Consider $Z= \sum_{j=1}^{a} \alpha_{j} Y_{j} + \sum_{k=1}^{b} \beta_{k} X_{k} \in {\mathfrak l}_{N}$ and
$\varphi \in L$. Conditions (3) and (4) imply
\begin{equation}
\label{equ:act}
 \varphi  \circ {\rm exp} \left( \sum_{j=1}^{a} \alpha_{j} Y_{j} + \sum_{k=1}^{b} \beta_{k} X_{k} \right) =
 {\rm exp}\left( \sum_{j=1}^{a} \tilde{\alpha}_{j} Y_{j} + \sum_{k=1}^{b} \beta_{k} X_{k} \right) \circ \varphi
\end{equation}
for some $\tilde{\alpha}_{1}, \hdots, \tilde{\alpha}_{a} \in \hat{K}_{n}$.
Let $Z'= \sum_{j=1}^{a} \alpha_{j}' Y_{j} + \sum_{k=1}^{b} \beta_{k}' X_{k} \in {\mathfrak l}_{N}$.
We have
\begin{equation}
\label{equ:bcha}
 {\rm exp}(Z) \circ {\rm exp}(Z')  = {\rm exp}
\left( \sum_{j=1}^{a} {\alpha}_{j}'' Y_{j} + \sum_{k=1}^{b} (\beta_{k}+\beta_{k}') X_{k} \right)
\end{equation}
where $\log ({\rm exp}(Z) \circ {\rm exp}(Z')) - (Z+Z')
\in \overline{{\mathfrak l}_{N}}^{(1)} \subset {\mathfrak l}'$
by Baker-Campbell-Hausdorff formula (\ref{equ:BCH}) and conditions (2) and (3).
The previous formulas   imply
\[ \left[ \varphi \circ  {\rm exp}\left( \sum_{j=1}^{a} {\alpha}_{j} Y_{j} + \sum_{k=1}^{b} \beta_{k} X_{k} \right),
\eta \circ \mathrm{exp} \left( \sum_{j=1}^{a} \alpha_{j}' Y_{j} + \sum_{k=1}^{b} \beta_{k}' X_{k} \right) \right]= \]
\begin{equation}
\label{equ:com}
[\varphi, \eta] \circ {\rm exp}\left( \sum_{j=1}^{a} {\alpha}_{j}'' Y_{j}  \right)
\end{equation}
for $\sum_{j=1}^{a} {\alpha}_{j} Y_{j} + \sum_{k=1}^{b} \beta_{k} X_{k}$,
$\sum_{j=1}^{a} \alpha_{j}' Y_{j} + \sum_{k=1}^{b} \beta_{k}' X_{k} \in {\mathfrak l}_{N}$
and $\varphi , \eta \in L$.

Since ${\mathfrak l}' \cap {\mathfrak l}_{N}$ is a Lie algebra of formal nilpotent vector fields that is
closed in the Krull topology, 
the group $\mathrm{exp}({\mathfrak l}' \cap {\mathfrak l}_{N})$
is a pro-algebraic subgroup of $L_{u}$ by Proposition \ref{pro:consag}.

Suppose $L/L_{u}$ is cyclic. There exists $\varphi_{0} \in G$ such that any element of $L$
is of the form $\varphi_{0}^{j} \circ \phi$ for some $j \in {\mathbb Z}$ and $\phi \in L_{u}$.
The formula (\ref{equ:com}) implies that
any commutator of elements of $L$ belongs to  $\mathrm{exp}({\mathfrak l}' \cap {\mathfrak l}_{N})$.
Since $\mathrm{exp}({\mathfrak l}' \cap {\mathfrak l}_{N})$ is pro-algebraic,
we deduce $\overline{L}^{(1)} \subset \mathrm{exp}({\mathfrak l}' \cap {\mathfrak l}_{N})$.

Suppose $L/L_{u}$ is a finite abelian group. Given $\varphi, \eta \in L$
there exists $p \in {\mathbb N}$ such that $\varphi^{p}$ and $\eta^{p}$ belong
to $L_{u}$. Since $[\varphi, \eta] \in L_{u}$ we have
$[\varphi, \eta] = \mathrm{exp} ( \sum_{j=1}^{a} {\alpha}_{j} Y_{j} + \sum_{k=1}^{b} \beta_{k} X_{k})$
and
\[ \varphi \circ \eta \circ \varphi^{-1} = \mathrm{exp}
\left( \sum_{j=1}^{a} {\alpha}_{j} Y_{j} + \sum_{k=1}^{b} \beta_{k} X_{k} \right) \circ \eta. \]
By iterating both terms of the previous equation $p$ times and using formulas
(\ref{equ:act}) and (\ref{equ:bcha}) we obtain
\[ \varphi \circ \eta^{p} \circ \varphi^{-1} = \mathrm{exp}
\left( \sum_{j=1}^{a} {\alpha}_{j}' Y_{j} + p \sum_{k=1}^{b} \beta_{k} X_{k} \right) \circ \eta^{p} \]
for some $\alpha_{1}', \hdots, \alpha_{a}' \in \hat{K}_{n}$. This leads us to
\[ \eta^{-p} \circ \varphi \circ \eta^{p}  = \mathrm{exp}
\left( \sum_{j=1}^{a} {\alpha}_{j}'' Y_{j} + p \sum_{k=1}^{b} \beta_{k} X_{k} \right) \circ \varphi \]
by equation (\ref{equ:act}). We get
\[ \eta^{-p} \circ \varphi^{p} \circ \eta^{p}  = \mathrm{exp}
\left( \sum_{j=1}^{a} {\alpha}_{j}''' Y_{j} + p^{2} \sum_{k=1}^{b} \beta_{k} X_{k} \right) \circ \varphi^{p}. \]
The infinitesimal generator of $[\eta^{-p},\varphi^{p}]$ belongs to ${\mathfrak l}'$
by equation (\ref{equ:com}).
Therefore we deduce
$[\varphi, \eta] = \mathrm{exp} ( \sum_{j=1}^{a} {\alpha}_{j} Y_{j}) \in \mathrm{exp}({\mathfrak l}' \cap {\mathfrak l}_{N})$
for all $\varphi, \eta \in L$. Since  $\mathrm{exp}({\mathfrak l}' \cap {\mathfrak l}_{N})$ is
pro-algebraic, $\overline{L}^{(1)}$ is contained in $\mathrm{exp}({\mathfrak l}' \cap {\mathfrak l}_{N})$.

The group $\overline{L}^{(1)}$ is   pro-algebraic
by Proposition \ref{pro:proa}. Since $\overline{L}^{(1)}$ is unipotent,
the Lie algebra of $\overline{L}^{(1)}$
consists of the infinitesimal generators of elements of $\overline{L}^{(1)}$
by Corollary \ref{cor:lie}. Thus it
is contained in ${\mathfrak l}' \cap {\mathfrak l}_{N}$.
\end{proof}
\begin{pro}
\label{pro:main2}
$\psi (2) \leq 5$ (cf. Definition \ref{def:psi}).
\end{pro}
\begin{proof}
We can suppose
$G = \overline{G}^{(0)}$ without lack of generality
by Lemma \ref{lem:solaux}.
Denote $\ell=\ell (G)$ and $\ell_{u}=\ell (G_{u})$.
We are going to prove the proposition by considering different cases.

Notice that we always have $\ell (G) \leq \sigma (2) + \ell_{u}$ by
Proposition \ref{pro:ibound}. Since $\sigma (2)=4$ we can suppose
$\ell_{u} \geq 2$.
We claim $\kappa (\ell_{u}-1) = 1$. Otherwise we have $\kappa (0)=2$
and Corollary \ref{cor:maxd} implies $\kappa (1)=0$ and
$\ell_{u} = 1$.

We denote ${\mathcal M}_{j}={\mathcal M}(\overline{\mathfrak g}_{N}^{(j)})$.
Let $Y \in \overline{\mathfrak g}_{N}^{(\ell_{u}-1)}$
with $Y \not \equiv 0$.
Any element $\varphi \in G$ preserves
$ \overline{\mathfrak g}_{N}^{(\ell_{u}-1)}$ by Lemma \ref{lem:normal}.
We deduce that
$g \circ \varphi \in {\mathcal M}_{\ell_{u}-1} $ for all
$g \in {\mathcal M}_{\ell_{u}-1}$ and $\varphi \in G$.
If ${\mathcal M}_{\ell_{u}-1} \neq {\mathbb C}$ then the formal foliation
defined by ${\mathcal M}_{\ell_{u}-1}$ has codimension $1$ since
$Y(g) \equiv 0$ for any $g \in {\mathcal M}_{\ell_{u}-1}$. Thus
we obtain
$g \circ \varphi =g $ for all
$g \in {\mathcal M}_{\ell_{u}-1}$ and $\varphi \in \overline{G}^{(3)}$
by Proposition \ref{pro:cod1}. The same result obviously holds if ${\mathcal M}_{\ell_{u}-1} = {\mathbb C}$.

The property
$\varphi_{*}  \overline{\mathfrak g}_{N}^{(\ell_{u}-1)} =  \overline{\mathfrak g}_{N}^{(\ell_{u}-1)}$
implies that $\varphi^{*} Y = f Y$ for some $f \in \hat{K}_{2}$.
Indeed $f$ belongs to ${\mathcal M}_{\ell_{u}-1}$
since
\[ [ \overline{\mathfrak g}^{(\ell_{u}-1)}, \overline{\mathfrak g}^{(\ell_{u}-1)}]=0 \implies
[Y,fY] = Y(f) Y=0. \]
Suppose $\varphi^{*} Y = f Y$ and $\eta^{*} Y = g Y$. The equation
\[ [\varphi, \eta]^{*} Y= (\eta^{-1})^{*} (\varphi^{-1})^{*} \eta^{*} \varphi^{*} Y =
\frac{(f \circ \eta \circ \varphi^{-1} \circ \eta^{-1}) (g \circ \varphi^{-1} \circ \eta^{-1})}{(f \circ  \varphi^{-1} \circ \eta^{-1}) (g \circ \eta^{-1})}  Y \]
and the invariance of  ${\mathcal M}_{\ell_{u}-1}$ by the action induced by $\overline{G}^{(3)}$
implies $\varphi^{*} Y \equiv Y$ for any $\varphi \in \overline{G}^{(4)}$.

Let us consider the case $\kappa(0)=1$.
We have $\overline{G}^{(4)} \subset G_{u}$ since $\sigma (2)=4$.
Fix $\varphi \in \overline{G}^{(4)}$. It is of the form
${\rm exp}(h Y)$ for some $h Y \in \hat{\mathfrak X}_{N} \cn{2}$.
Since $\varphi^{*} Y = Y$, Lemma \ref{lem:pol} implies $Y(h)=0$.
We deduce $[\log \eta, \log \omega]=0$ for all $\eta, \omega \in \overline{G}^{(4)}$.
Therefore $\overline{G}^{(4)}$ is a commutative group.
%
%
We obtain $\ell (G) \leq 5$.

Finally let us consider the case $\kappa(0)=2$.
Let $X \in {\mathfrak g}_{N}$ such that $\{ Y,X \}$
is a base of ${\mathfrak g}_{N} \otimes_{\mathbb C} \hat{K}_{2}$.
Let $W \in {\mathfrak g}_{N}$.
The formal vector field $W$ is of the form $\alpha Y + \beta X$.
Since $\kappa(1)<2$ by Corollary \ref{cor:maxd}, Lemma \ref{lem:aux}
implies that $\beta$ belongs to ${\mathcal M}_{0} = {\mathbb C}$.
Hence given $\upsilon \in G$ we have $\upsilon^{*} X = \gamma Y + t X$ for some
$\gamma \in \hat{K}_{2}$ and $t \in {\mathbb C}^{*}$.
Since $\upsilon^{*} Y$ is contained in $\overline{\mathfrak g}_{N}^{\ell_{u}-1}$,
we obtain $\upsilon^{*} X = \gamma Y + X$ for
any $\upsilon \in \overline{G}^{(1)}$ and some
$\gamma \in \hat{K}_{2}$ depending on $\upsilon$.
We denote $L=\overline{G}^{(3)}$.
The group $L/L_{u}$ is isomorphic to either $\{Id\}$ or $\{Id, -Id\}$
by Corollary \ref{cor:lin2}.

We apply Lemma \ref{lem:BCH} to
$L=\overline{G}^{(3)}$, $Y_{1}=Y$, $X_{1}=X$;
hence any element $\varphi \in \overline{G}^{(4)}$ is of the form
${\rm exp}(\alpha Y)$ for some $\alpha \in \hat{K}_{2}$ and
$\alpha Y$ in $\hat{\mathfrak X}_{N} \cn{2}$.
We obtain $\ell (G) \leq 5$ analogously as for $\kappa (0)=1$.
\end{proof}
\section{Soluble lengths of groups in dimension $3$, $4$ and $5$.}
\label{sec:35}
\begin{pro}
\label{pro:main3}
$\psi (n) \leq 2n+1$ (cf. Definition \ref{def:psi}) for any $n \in \{3,4,5\}$.
%
\end{pro}
The rest of this section is devoted to show the previous proposition.
We divide the proof in several cases. We can suppose
$G = \overline{G}^{(0)}$ by Lemma \ref{lem:solaux}.
Notice that the function $\sigma$ is always equal to $5$ for
$n \in \{3,4,5\}$ (cf. \cite{Newman}).
\subsection{There is no $p$ such that $\kappa (p)=n-1$}
\label{subsec:nop2}
Corollary \ref{cor:maxd} and Proposition \ref{pro:intd}
imply $\ell (G_{u}) \leq 2n-3$.
If $\ell (G_{u}) \leq 2n-4$ then $\ell (G) \leq \sigma(n) + 2n-4 = 5 + 2n-4=2n+1$
by Proposition \ref{pro:ibound}.
Thus we can suppose $\ell (G_{u})=2n-3$.
Again we use Corollary \ref{cor:maxd} and Proposition \ref{pro:intd}
to obtain $\kappa(0)=n$, $\kappa(1)=n-2$, $\kappa(2)=n-2$ and $\kappa(3)=n-3$.
Indeed we have $\kappa (2j)=n-j-1$ and $\kappa (2j-1)=n-j-1$
for $1 \leq j \leq n-2$.
Let $\{Y_{1}, \hdots, Y_{n-2} \} \subset \overline{\mathfrak g}_{N}^{(2)}$
be a base of $\overline{\mathfrak g}_{N}^{(2)} \otimes_{\mathbb C} \hat{K}_{n}$
whereas $X_{1},X_{2}$ are elements of ${\mathfrak g}_{N}$ such that
$\{Y_{1},\hdots,Y_{n-2}, X_{1},X_{2}\}$ is a base of the $\hat{K}_{n}$-vector space
${\mathfrak g}_{N}  \otimes_{\mathbb C} \hat{K}_{n}$.
Since ${\mathcal M}_{0} = {\mathbb C}$, Lemma \ref{lem:aux} implies that
any element $\sum_{j=1}^{n-2} \alpha_{j} Y_{j} + \beta X_{1} + \gamma X_{2}$ of
${\mathfrak g}_{N}$ where $\alpha_{1}, \hdots, \alpha_{n-2}, \beta, \gamma \in \hat{K}_{n}$
satisfies $\beta,\gamma \in {\mathbb C}$. Given $\varphi \in G$
the property $\varphi^{*} \overline{\mathfrak g}_{N}^{(2)} =   \overline{\mathfrak g}_{N}^{(2)}$
implies
\[
\left(
\begin{array}{c}
\varphi^{*} X_{1} \\
\varphi^{*} X_{2}
\end{array}
\right) = A_{\varphi} \left( \begin{array}{c}
X_{1} \\
X_{2}
\end{array} \right) + \left(
\begin{array}{c}
\sum_{j=1}^{n-2} g_{j} Y_{j} \\
\sum_{j=1}^{n-2} h_{j} Y_{j}
\end{array}
\right)
\]
where $A_{\varphi} \in \mathrm{GL}(2,{\mathbb C})$.
Since $\rho(2)=4$ we deduce that
\begin{equation}
\label{equ:res2}
\left(
\begin{array}{c}
\varphi^{*} X_{1} \\
\varphi^{*} X_{2}
\end{array}
\right) =  \left( \begin{array}{c}
X_{1} \\
X_{2}
\end{array} \right) + \left(
\begin{array}{c}
\sum_{j=1}^{n-2} g_{j}' Y_{j} \\
\sum_{j=1}^{n-2} h_{j}' Y_{j}
\end{array}
\right)
\end{equation}
for any $\varphi \in \overline{G}^{(4)}$.

 We denote $L= \overline{G}^{(4)}$. The group $L/L_{u}$ is finite abelian by
Corollary \ref{cor:fag}.
Notice that $[X_{1},X_{2}]$ is of the form $\sum_{j=1}^{n-2} f_{j} Y_{j}$
for some $f_{1},\hdots,f_{n-2} \in \hat{K}_{n}$
since $\kappa (1)=n-2$.
We apply Lemma \ref{lem:BCH} to
$L=\overline{G}^{(4)}$. We obtain that any
element $\varphi$ of $\overline{G}^{(5)}$ is of the form
${\rm exp}(\sum_{j=1}^{n-2} \alpha_{j} Y_{j})$ for some
$\alpha_{1},\hdots,\alpha_{n-2} \in \hat{K}_{n}$
and $\sum_{j=1}^{n-2} \alpha_{j} Y_{j} \in \hat{\mathfrak X}_{N} \cn{n}$.
Hence the Lie algebra ${\mathfrak h}$ of $\overline{G}^{(5)}$
satisfies $\dim ({\mathfrak h} \otimes_{\mathbb C} \hat{K}_{n}) \leq n-2$.
Remark \ref{rem:dim1}
implies $\ell (G) \leq 5 + 2(n-2) = 2n+1$.
\subsection{There exists $p$ such that $\kappa(p)=n-1$ but no $q$ with
$\kappa(q)=n-2$.}
We obtain $\ell (G_{u}) \leq 2n-3$ by Corollary \ref{cor:maxd} and Proposition \ref{pro:intd}.
If $\ell (G_{u}) \leq 2n-4$ then $\ell (G) \leq 2n+1$ by Proposition \ref{pro:ibound}.
Hence we can suppose that $\ell (G_{u})=2n-3$.
We obtain $\kappa(0)=n$, $\kappa(1)=n-1$, $\kappa(2)=n-1$, $\kappa(3)=n-3$,
$\kappa (4)=n-3$ and so on.
We choose a base $\{ Y_{1}, \hdots, Y_{n-3} \}$ of
${\mathfrak l}' := \overline{\mathfrak g}_{N}^{(4)} \otimes_{\mathbb C} \hat{K}_{n}$
contained in $\overline{\mathfrak g}_{N}^{(4)}$.
We choose $X_{1}, X_{2} \in \overline{\mathfrak g}_{N}^{(2)}$
such that $\{ Y_{1}, \hdots, Y_{n-3}, X_{1}, X_{2} \}$ is a base of
$\overline{\mathfrak g}_{N}^{(2)} \otimes_{\mathbb C} \hat{K}_{n}$.
We have ${\mathcal M}_{1} \neq {\mathbb C}$; otherwise
Proposition \ref{pro:maxd} would imply $\kappa(2) < \kappa(1)$.
Moreover $[X_{1},X_{2}]$ belongs to ${\mathfrak l}'$
(or vanishes if $n=3$) since $[X_{1},X_{2}] \in \overline{\mathfrak g}_{N}^{(3)}$
and $\kappa (3)=n-3$.

Let $\varphi \in G$. Since
$\varphi^{*} \overline{\mathfrak g}_{N}^{(2)} = \overline{\mathfrak g}_{N}^{(2)}$
by Lemma \ref{lem:normal}, we obtain
\[ \varphi^{*} X_{j} = \alpha_{j} X_{1} + \beta_{j} X_{2} + \sum_{k=1}^{n-3} \gamma_{j,k} Y_{k} \]
where $j \in \{1,2\}$ and the coefficients belong to $\hat{K}_{n}$.
We have $[\overline{\mathfrak g}_{N}^{(2)},\overline{\mathfrak g}_{N}^{(2)}] \subset
\overline{\mathfrak g}_{N}^{(3)}$ by definition and
$[\overline{\mathfrak g}_{N}^{(2)},\overline{\mathfrak g}_{N}^{(4)}] \subset
\overline{\mathfrak g}_{N}^{(4)}$ by Lemma \ref{lem:normal}.
In particular we obtain
$[\varphi^{*} X_{j}, X_{k}] \in {\mathfrak l}'$ and
$[\varphi^{*} X_{j}, Y_{l}] \in {\mathfrak l}'$ for all
$1 \leq j,k \leq 2$ and $1 \leq l \leq n-3$.
This implies
\[ X_{1} (\alpha_{j})=X_{2}(\alpha_{j})=Y_{1}(\alpha_{j}) = \hdots = Y_{n-3}(\alpha_{j})=0 \]
and analogously for $\beta_{j}$ for $j \in \{1,2\}$. Thus
$\alpha_{j}$ and  $\beta_{j}$ belong to ${\mathcal M}_{2}$ for any $j \in \{1,2\}$.
We obtain
\[
\left(
\begin{array}{c}
\varphi^{*} X_{1} \\
\varphi^{*} X_{2}
\end{array}
\right) = A_{\varphi} \left( \begin{array}{c}
X_{1} \\
X_{2}
\end{array} \right) + \left(
\begin{array}{c}
\sum_{k=1}^{n-3} \gamma_{1,k} Y_{k} \\
\sum_{k=1}^{n-3} \gamma_{2,k} Y_{k}
\end{array}
\right) \]
where $A_{\varphi} \in \mathrm{GL} (2, {\mathcal M}_{2})$.

Since ${\mathcal M}_{1} \neq {\mathbb C}$ and $\kappa (1)=n-1$,
the foliation defined by ${\mathcal M}_{1}$ has codimension $1$.
We have $g \circ \varphi = g$ for all
$g \in {\mathcal M}_{1}={\mathcal M}_{2}$ and $\varphi \in \overline{G}^{(3)}$
by Proposition \ref{pro:cod1}.
Hence the mapping $\Lambda: \overline{G}^{(3)} \to \mathrm{GL} (2, {\mathcal M}_{2})$
defined by $\Lambda(\varphi) = A_{\varphi}$
is a homomorphism of groups.

Let us show $\ell (\Lambda( \overline{G}^{(3)})) \leq 3$.
We have $\Lambda( \overline{G}^{(3)})^{(3)} \subset \{-Id, Id\}$ by
Corollary \ref{cor:lin2}. Given $\varphi \in \overline{G}^{(6)}$ we obtain
$\varphi^{t} \in \overline{G}^{(6)}$ for any $t \in {\mathbb C}$
since $\overline{G}^{(6)}$ is pro-algebraic, $\overline{G}^{(6)} \subset G_{u}$ and
Lemma \ref{lem:solaux}. Given $j \in \{1,2\}$ we obtain
\[ (\varphi^{t})^{*} X_{j} = \mu_{j}(t) X_{j} + \sum_{k=1}^{n-3} \gamma_{j,k,t} Y_{k} \]
where $\mu_{j}: {\mathbb C} \to \{-1,1\}$ is a continuous function such that $\mu_{j}(0)=1$.
Thus $\mu_{j}$ is constant and equal to $1$ for $j \in \{1,2\}$.

Fix $\varphi \in \overline{G}^{(6)}$.
The previous discussion implies
$\varphi^{*} X_{j} -X_{j} \in {\mathfrak l}'$
for $j \in \{1,2\}$ and
$g \circ \varphi = g$ for any $g \in {\mathcal M}_{1}$.
Since $\overline{G}^{(5)} \subset G_{u}$ Proposition \ref{pro:lieder} implies
$\log \varphi \in \overline{\mathfrak g}_{N}^{(1)}$.
We deduce
$\log \varphi = \alpha_{1} X_{1} + \alpha_{2} X_{2} + \sum_{k=1}^{n-3} \delta_{k} Y_{k}$.
It is clear that $(\varphi^{t})^{*} Y_{k} \in {\mathfrak l}'$
for any $t \in {\mathbb C}$ since
$\eta^{*} \overline{\mathfrak g}_{N}^{(4)} = \overline{\mathfrak g}_{N}^{(4)}$
for any $\eta \in G$.
Since $[\log \varphi, Z] = \lim_{t \to 0} ((\varphi^{t})^{*} Z  - Z)/t$, we obtain
that $[X_{1}, \log \varphi]$, $[X_{2}, \log \varphi]$ and $[Y_{k},\log \varphi]$
belong to ${\mathfrak l}'$ for any $1 \leq k \leq n-3$.
We deduce that $\alpha_{1}$ and $\alpha_{2}$ belong to ${\mathcal M}_{2}$.
Therefore $\log [\varphi, \eta]$ belongs to ${\mathfrak l}'$ for all
$\varphi, \eta \in \overline{G}^{(6)}$
by Baker-Campbell-Hausdorff formula. Another application of
Baker-Campbell-Hausdorff provides that the Lie algebra ${\mathfrak h}$
of $\overline{G}^{(7)}$ is contained in ${\mathfrak l}'$
and in particular $\dim ({\mathfrak h} \otimes_{\mathbb C} \hat{K}_{n}) \leq n-3$.
We obtain $\ell (G) \leq 7 + 2(n-3)=2n+1$ by Remark \ref{rem:dim1}.
\subsection{There exists $p,q$ such that $\kappa(p)=n-1$,
$\kappa(q)=n-2$ and ${\mathcal M}_{p} = {\mathbb C}$.}
Let us remark that there exists a unique $p \in {\mathbb N} \cup \{0\}$
such that $\kappa(p)=n-1$ by Proposition \ref{pro:maxd}.
We can suppose that $\kappa (m) < n-2$ for $m>q$.
Let $\{Y_{1}, \hdots, Y_{n-2} \} \subset \overline{\mathfrak g}_{N}^{(q)}$ be a
base of ${\mathfrak l}':=\overline{\mathfrak g}_{N}^{(q)} \otimes_{\mathbb C} \hat{K}_{n}$.
We choose $X_{1} \in \overline{\mathfrak g}_{N}^{(p)}$ such that
$\{Y_{1},\hdots,Y_{n-2}, X_{1} \}$ is a base of
$\overline{\mathfrak g}_{N}^{(p)} \otimes_{\mathbb C} \hat{K}_{n}$.
\begin{lem}
\label{lem:cte2}
$[X_{1},W] \in {\mathfrak l}'$  for any $W \in {\mathfrak g}_{N}$.
\end{lem}
\begin{proof}
If $\kappa(0) \neq n$ then $\kappa(0)=n-1$ and $\kappa(1)=n-2$.
Moreover $\{ Y_{1},\hdots,Y_{n-2} \}$ is a base of
$\overline{\mathfrak g}_{N}^{(1)} \otimes_{\mathbb C} \hat{K}_{n}$.
The Lie bracket $[X_{1},W]$ belongs to $\overline{\mathfrak g}_{N}^{(1)}$
and then to ${\mathfrak l}'$.
Suppose $\kappa(0)=n$ from now on; it implies $p=1$.
We choose
$Z \in {\mathfrak g}_{N}$ such that
$\{Y_{1},\hdots,Y_{n-2},X_{1},Z\}$ is a base of
${\mathfrak g}_{N} \otimes_{\mathbb C} \hat{K}_{n}$.

The elements of $\overline{\mathfrak g}_{N}^{(p)}$ are of the form
$\sum_{j=1}^{n-2} \alpha_{j} Y_{j} + \beta X_{1}$ where $\beta \in {\mathcal M}_{p} = {\mathbb C}$ by
Lemma \ref{lem:aux}. The formal vector field $[X_{1},W]$ belongs to
$[\overline{\mathfrak g}_{N}^{(p)}, {\mathfrak g}_{N}] \subset\overline{\mathfrak g}_{N}^{(p)}$.
Hence $[X_{1},W]= \sum_{j=1}^{n-2} h_{j} Y_{j} + t X_{1}$ where   $t \in {\mathbb C}$.
Moreover $[Y_{j},W]$ belongs to ${\mathfrak l}'$ for $1 \leq j \leq n-2$
since $[Y_{j},W]$ belongs to $\overline{\mathfrak g}_{N}^{(q)}$.
We define $X_{2} = [X_{1},W]$ and $X_{j+1} = [X_{j},W]$ for $j \geq 2$.
Then $X_{j}$ is of the form $\sum_{k=1}^{n-2} h_{j,k} Y_{k} + t^{j-1} X_{1}$ for any
$j \in {\mathbb N}$.

The Lie algebra of $k$-jets of ${\mathfrak g}_{N}$ is ${\mathfrak g}_{k,u}$
by definition. Since ${\mathfrak g}_{N}$ is solvable,
${\mathfrak g}_{k,u}$ is a solvable Lie algebra of nilpotent matrices
for any $k \in {\mathbb N}$.
Therefore ${\mathfrak g}_{k,u}$ is nilpotent by Lie's Theorem \ref{teo:Lie}
for any $k \in {\mathbb N}$.
%
%
Hence the sequence $X_{j}$ tends to $0$ in the Krull topology.
This implies $t=0$.
\end{proof}
\begin{lem}
$\kappa(0)=n-1$.
In particular we have $\ell (G_{u}) \leq 2n-3$.
\end{lem}
\begin{proof}
Suppose $\kappa(0)=n$. We use the notations in Lemma \ref{lem:cte2}.
Let $W:=\sum_{j=1}^{n-2} \alpha_{j} Y_{j} + \beta X_{1} + \gamma Z \in {\mathfrak g}_{N}$.
We have
\[ [Y_{j}, W] \in {\mathfrak l}' \ {\rm and} \ [X_{1}, W] \in {\mathfrak l}'  \]
for any $1 \leq j \leq n-2$ by
$[\overline{\mathfrak g}_{N}^{(q)}, \overline{\mathfrak g}_{N}] \subset \overline{\mathfrak g}_{N}^{(q)}$
and Lemma \ref{lem:cte2}.
We deduce that $\beta, \gamma$ belong to ${\mathcal M}_{p}={\mathbb C}$.
As a consequence we obtain $\overline{\mathfrak g}_{N}^{(1)} \subset {\mathfrak l}'$
and $\dim (\overline{\mathfrak g}_{N}^{(1)} \otimes_{\mathbb C} \hat{K}_{n}) \leq n-2$.
We obtain a contradiction since there is no $p$ such that $\kappa(p)=n-1$.

We have $\kappa(1)=n-2$ by Proposition \ref{pro:maxd}.
Thus Proposition \ref{pro:lieder} and
Remark \ref{rem:dim1} imply   $\ell (G_{u}) \leq 1 + 2(n-2) = 2n-3$.
\end{proof}

Any element $W$ of ${\mathfrak g}_{N}$ is of the form
$\sum_{j=1}^{n-2} \alpha_{j} Y_{j} + s X_{1}$ for some
$s \in {\mathbb C}$ by Lemma \ref{lem:aux}.
Let  $\varphi \in G$. It satisfies
$\varphi^{*} {\mathfrak g}_{N} = {\mathfrak g}_{N}$ and then
$\varphi^{*} X_{1}$ is of the form  $\sum_{j=1}^{n-2} h_{j} Y_{j} + s X_{1}$ for some
$s \in {\mathbb C}^{*}$.
We obtain
$\varphi^{*} X_{1} - X_{1} \in {\mathfrak l}'$ for any $\varphi \in  \overline{G}^{(1)}$.
The group $\overline{G}^{(4)}/(\overline{G}^{(4)})_{u}$ is finite abelian by
Corollary \ref{cor:fag}.
The Lie algebra of the unipotent pro-algebraic group $\overline{G}^{(5)}$
is contained in ${\mathfrak l}'$ by Lemma
\ref{lem:BCH} applied to $L=\overline{G}^{(4)}$.
We obtain $\ell (G) \leq 2n+1$ by Remark \ref{rem:dim1}.
\subsection{There exists $p,q$ such that $\kappa(p)=n-1$,
$\kappa(q)=n-2$ and ${\mathcal M}_{p} \neq {\mathbb C}$.}
Let $p$ and $q$ such that $\kappa(m) <n-1$ if $m>p$
and $\kappa (l) < n-2$ if $l > q$.
Let $\{Y_{1},\hdots,Y_{n-2} \} \subset \overline{\mathfrak g}_{N}^{(q)}$
be a base of ${\mathfrak l}':=\overline{\mathfrak g}_{N}^{(q)} \otimes_{\mathbb C} \hat{K}_{n}$.
We choose $X_{1} \in \overline{\mathfrak g}_{N}^{(p)}$ such that
$\{Y_{1},\hdots,Y_{n-2}, X_{1} \}$ is a base of
$\overline{\mathfrak g}_{N}^{(p)} \otimes_{\mathbb C} \hat{K}_{n}$.

The foliation defined by ${\mathcal M}_{p}$ has codimension $1$
since $\kappa (p)=n-1$ and ${\mathcal M}_{p} \neq {\mathbb C}$.
Thus $g \circ \varphi = g$ for all $g \in {\mathcal M}_{p}$ and
$\varphi \in \overline{G}^{(3)}$ by Proposition \ref{pro:cod1}.

The elements of $\overline{\mathfrak g}_{N}^{(p)}$ are of the form
$\sum_{j=1}^{n-2} \alpha_{j} Y_{j} + \beta X_{1}$ where $\beta \in {\mathcal M}_{p}$
by Lemma \ref{lem:aux}.
Any $\varphi \in G$ satisfies
$\varphi^{*} X_{1}=  \sum_{j=1}^{n-2} \alpha_{j}' Y_{j} + \beta' X_{1}$ where  $\beta' \in {\mathcal M}_{p}$.
Thus $\varphi^{*} X_{1}$ is of the form $ \sum_{j=1}^{n-2} \alpha_{j}'' Y_{j} + X_{1}$ for any
$\varphi \in \overline{G}^{(4)}$.
Notice that $\overline{G}^{(4)}/(\overline{G}^{(4)})_{u}$ is a finite abelian group by
Corollary \ref{cor:fag}.

Fix $Z \in {\mathfrak g}_{N}$ with $\mathrm{exp} (Z) \in \overline{G}^{(4)} \cap G_{u}$.
The property $g \circ {\rm exp} (Z) = g$ implies
$Z(g)=0$ for any $g \in {\mathcal M}_{p}$ by Lemma \ref{lem:pol}.
Given $g \in \hat{K}_{n} \setminus {\mathbb C}$, the $\hat{K}_{n}$-vector subspace of
$\hat{K}_{n} \otimes_{\mathbb C} \hat{\mathfrak X} \cn{n}$
of elements having $g$ as first integral has dimension $n-1$.
In particular we deduce
\[ \{ W \in  \hat{K}_{n} \otimes_{\mathbb C} \hat{\mathfrak X} \cn{n} : W(g) =0 \ \forall g \in {\mathcal M}_{p} \} =
\overline{\mathfrak g}_{N}^{(p)} \otimes_{\mathbb C} \hat{K}_{n} . \]
Thus $Z$ is of the form $\sum_{j=1}^{n-2} \alpha_{j} Y_{j} + \beta X_{1}$.
The one parameter group $\{ \mathrm{exp}(t  Z): t \in {\mathbb C} \}$ is contained in
$\overline{G}^{(4)} \cap G_{u}$ since $\overline{G}^{(4)}$ is pro-algebraic by
Proposition \ref{pro:proa}. Therefore
${\rm exp}(t Z)^{*}  Y_{j} - Y_{j}$ and ${\rm exp}(t Z)^{*}  X_{1} - X_{1}$
belong to ${\mathfrak l}'$ for all $1 \leq j \leq n-2$ and $t \in {\mathbb C}$.
We deduce
$[Y_{j},Z] \in {\mathfrak l}'$ and $[X_{1},Z] \in {\mathfrak l}'$
for any $1 \leq j \leq n-2$.
Hence $\beta$ belongs to ${\mathcal M}_{p}$.
%
Any element $\varphi \in \overline{G}^{(5)}$ is of the form
${\rm exp}(\sum_{j=1}^{n-2} \gamma_{j} Y_{j})$ for some
$\sum_{j=1}^{n-2} \gamma_{j} Y_{j} \in {\mathfrak g}_{N}$
by Lemma \ref{lem:BCH} applied to $L=\overline{G}^{(4)}$.
We obtain $\ell (G) \leq 2n+1$ by Remark \ref{rem:dim1}.
\begin{rem}
\label{rem:impound}
It is easy to see that, up to minor changes,
Proposition \ref{pro:main3} implies $\ell (G) \leq \sigma (n) + 2n -4$ for
every solvable subgroup $G$ of $\diffh{}{n}$ with $n \geq 3$.
Next section provides an example of a subgroup $G^{n}$ of
$\mathrm{Diff} ({\mathbb C}^{n},0)$ such that $\ell (G^{n}) =2n+1$
for any $n \geq 2$. We deduce
\[ 2n+1 \leq \psi(n) \leq  \sigma (n) + 2n -4 \ \ \forall n \geq 3. \]
The inequality $2n+1 \leq \sigma (n) + 2n -4$ for $n \geq 3$ is an equality
exactly for $n=3,4,5$, providing $\psi (n)=2n+1$ for $n \in \{3,4,5\}$.
The next best result is obtained for $n \in \{6,7\}$ since
$\sigma (6)=\sigma (7)=6$.  In fact we obtain
$(\psi(6), \psi(7)) \in \{ (13,15), (13,16), (14,16) \}$ by Proposition \ref{pro:mas2}.
\end{rem}
\section{Examples}
\label{sec:examples}
We provide examples in \cite{JR:arxivdl} of solvable groups
$G \subset \diff{}{n}$ such that $j^{1} G$ is a connected group of
matrices and $\ell (G)=2n$ for any $n \in {\mathbb N}$.
Thus a solvable subgroup $G$ of $\diffh{}{n}$ ($2 \leq n \leq 5$)
of greatest derived length satisfies
$\ell (G) \in \{2n, 2n+1 \}$.  In this section we prove that
the upper estimates are the correct ones.

Let $L$ be a solvable group of $\mathrm{GL} (2,{\mathbb C})$ such that
$\ell (L)=\rho(2)=4$. We denote
\[ \phi_{\lambda,\mu} = {\rm exp} \left( (\lambda x + \mu y)
\left( x \frac{\partial}{\partial x} + y \frac{\partial}{\partial y} \right)   \right) \]
for $\lambda, \mu \in {\mathbb C}$. We define the group
\[ N = \left\{  \phi_{\lambda,\mu} :
\lambda,\mu \in {\mathbb C}  \right\} . \]
We consider the group $G^{2} = N \rtimes L$ obtained as semidirect product of
$N$ and $L$. We have
\[ G^{2} = \left\{ T \circ \phi_{\lambda,\mu} :
T \in L, \ \lambda,\mu \in {\mathbb C}  \right\} . \]
The group $N$ is a normal commutative subgroup of $G^{2}$.
\begin{pro}
\label{pro:lg25}
$\ell (G^{2})=5$ and $(G^{2})^{(4)} = N$.
\end{pro}
\begin{proof}
We denote $G=G^{2}$.
The group $G^{(4)}$ is contained in the commutative group $N$.
Thus $G$ is solvable. We have $4 \leq \ell (G) \leq 5$.
It suffices to prove that $G^{(4)} \neq \{Id\}$.

Consider $T \in L$ and $\phi_{\lambda,\mu} \in N$. We have
\[ [T^{-1}, \phi_{\lambda,\mu}] =
T^{-1} \circ \phi_{\lambda,\mu} \circ T \circ \phi_{-\lambda,-\mu} =
\phi_{\lambda',\mu'} \]
where $\lambda' x + \mu' y = (\lambda x + \mu y) \circ T - (\lambda x + \mu y)$.
Indeed we obtain
\[
\left(
\begin{array}{c}
\lambda' \\
\mu'
\end{array}
\right) =
(T^{t} - Id) \left(
\begin{array}{c}
\lambda \\
\mu
\end{array}
\right)
\]
where $T^{t}$ is the transposed matrix of $T$.

Denote $N_{j} = G^{(j)} \cap N$ for $j \geq 0$.
The previous discussion implies that if $N_{j}=N$ and $L^{(j)}$
contains a matrix with no eigenvalue equal to $1$ then
$N_{j+1}=N$. Since $L^{(3)} = \{Id, -Id\}$ by Corollary \ref{cor:lin2}, we obtain
$G^{(4)} = N_{4} = N$.
\end{proof}
\begin{rem}
For the sake of clarity let us give an example of subgroup $L$
of $\mathrm{GL}(2,{\mathbb C})$ such that $\ell (L)=4$.
We define the group $L$ generated by
\[ \frac{1}{\sqrt{2}} \left(
\begin{array}{cc}
1+i & 0 \\
0 & 1-i
\end{array}
\right) \ \ \mathrm{and} \ \
\frac{1}{2} \left(
\begin{array}{rr}
1+ i & -1 + i  \\
1 + i & 1 - i
\end{array}
\right) , \]
cf. \cite{Shurman}[chapter 2, section 7, p. 44].
The group $L$ is a subgroup of $\mathrm{SL}(2, {\mathbb C})$.
It is a finite group with $48$ elements.
The projection $PL$ of $L$ in $\mathrm{PSL} (2, {\mathbb C})$ is a representation
of the group of orientation-preserving isometries of the cube (or the octahedron).
Moreover $PL$ is isomorphic to the group $S_{4}$
of permutations of $4$ elements.
\end{rem}
Next, we show the existence of an example of a solvable subgroup of
$\diff{}{n}$ whose length is equal to $2n+1$ for any $n \geq 2$.
\begin{pro}
\label{pro:mas2}
Let $H$ be a solvable subgroup of $\mathrm{Diff}({\mathbb C}^{n},0)$ with $n \geq 1$.
Then there exists a solvable subgroup $G$ of $\mathrm{Diff}({\mathbb C}^{n+1},0)$
such that $\ell (G) = \ell (H) + 2$. In particular we obtain
$\psi (n+1) \geq \psi(n) +2$ for any $n \in {\mathbb N}$.
\end{pro}
\begin{proof}
The result is trivial if $H$ is the trivial group since
\[ \left\{ \left( x_{1},\hdots,x_{n}, \frac{\lambda x_{n+1}}{1 + t x_{n+1}} \right) :
\lambda \in {\mathbb C}^{*} \ \mathrm{and} \ t \in {\mathbb C} \right\} \]
is a subgroup of $\mathrm{Diff}({\mathbb C}^{n+1},0)$ of derived length $2$.
We suppose $H \neq \{Id\}$ from now on.

We define
\[ \chi_{a,b} = (x_{1}, \hdots, x_{n}, a(x_{1},\hdots,x_{n}) x_{n+1} + b(x_{1},\hdots,x_{n})) \]
where $a(0) \neq 0$ and $b (0)=0$.
Consider the group
\[ G = \{ (\phi (x_{1},\hdots,x_{n}), a(x_{1},\hdots,x_{n}) x_{n+1} + b(x_{1},\hdots,x_{n})) : \]
\[ \ \phi \in H, \ a \in {\mathcal O}_{n} \setminus {\mathfrak m} \
\mathrm{and} \ b \in {\mathcal O}_{n} \cap {\mathfrak m}  \}. \]
The group $G^{(\ell (H))}$ is contained in
$\{ \chi_{a,b} : a \in {\mathcal O}_{n} \setminus {\mathfrak m}, \
b \in {\mathcal O}_{n} \cap {\mathfrak m} \}$. We obtain
$\ell (G) \leq \ell (H) +2$.

Fix an element $\phi_{0} \in H^{(\ell (H) -1)} \setminus \{Id\}$.
Let $\vartheta$ be the element of $G$ defined by
\[ \vartheta (x_{1},\hdots,x_{n+1})= (\phi_{0}(x_{1},\hdots,x_{n}),x_{n+1}) . \]
It is clear that $\vartheta$ belongs to $G^{(l)}$ for any $0 \leq l < \ell (H)$.
We define the operator
$\Delta:  {\mathcal O}_{n} \cap {\mathfrak m} \to  {\mathcal O}_{n} \cap {\mathfrak m}$
by $\Delta (f) = f \circ \phi_{0} - f$.

 The equation
\[ [\vartheta^{-1}, \chi_{1,b}  ] =
\chi_{1,   b(x_{1},\hdots,x_{n}) \circ \phi_{0} - b(x_{1},\hdots, x_{n})} = \chi_{1, \Delta (b)} \]
implies that if $\chi_{1,b}$ belongs to $G^{(l)}$ for $l < \ell (H)$ then
$\chi_{1, \Delta (b)}$ belongs to $G^{(l+1)}$.
Suppose $\Delta^{k} \neq 0$ for any $k \geq 1$; we prove this property later on.
Then there exists an element $\chi_{1,b'}$ of $G^{(\ell(H))} \setminus \{Id\}$.

  Consider an element $a$ of ${\mathcal O}_{n}$ such that $a(0)=1$.
We denote by $\ln a$ the unique element of ${\mathfrak m}$ such that
$e^{\ln a} =a$. We have
\[  [\vartheta^{-1}, \chi_{e^{\ln a},0} ] = [\vartheta^{-1}, \chi_{a,0} ] =
\chi_{e^{\ln (a) \circ \phi_{0} - \ln (a)},0} = \chi_{e^{\Delta(\ln a)},0} . \]
Again the non-nilpotence of the operator $\Delta$ implies the existence of an element
$\chi_{a',0}$  in $G^{(\ell (H))}$ with $a' \not \equiv 1$.
The diffeomorphisms $\chi_{a',0}$ and $\chi_{1,b'}$ do not commute.
Therefore the group $G^{(\ell (H)+1)}$ is non-trivial.
Since $\ell (G) \leq \ell (H) +2$ we deduce $\ell (G)=\ell (H)+2$.

Let us show that $\Delta^{k} \not \equiv 0$ for every $k \geq 1$.
It is clear for $k=1$ since $\phi_{0} \neq Id$ implies
$\Delta (x_{j}) \neq 0$ for some $1 \leq j \leq n$.
We have
\[ \Delta (f g) = (f g) \circ \phi_{0} - fg =
(f \circ \phi_{0} -f)  (g \circ \phi_{0} -g) + (f \circ \phi_{0} -f) g +
f (g \circ \phi_{0} -g) \]
and then $\Delta (f g) = \Delta (f) \Delta (g) + \Delta (f) g + f \Delta (g)$.
Given $k \geq 1$ we obtain
\begin{equation}
\label{equ:deltak}
\Delta^{k} (f g)= \sum_{k \leq m+l, \ 0 \leq m \leq k, \ 0 \leq l \leq k} c_{kml} \Delta^{m}(f) \Delta^{l}(g)
\end{equation}
where $c_{kml}$ is a positive integer number independent of $f$ and $g$ for
$k \leq m+l$, $0 \leq m \leq k$ and $0 \leq l \leq k$.
In order to show that $\Delta$ is not nilpotent it suffices to prove that
if $\Delta^{k}(f) \neq 0$ and $\Delta^{k+1}(f)=0$ then
$\Delta^{2k}(f^{2}) \neq 0$.
Equation (\ref{equ:deltak}) implies
$\Delta^{2k}(f^{2}) = c_{2k,k,k} \Delta^{k}(f)^{2}$ and hence
$\Delta^{2k}(f^{2}) \neq 0$.
\end{proof}
\begin{rem}
Consider the subgroup $G^{n}$
of $\mathrm{Diff}({\mathbb C}^{n},0)$ consisting of elements of the form
 \[
 (\phi (x_{1},x_{2}), \hdots,  a_{j}(x_{1},\hdots, x_{j-1}) x_{j} + b_{j}(x_{1},\hdots, x_{j-1}), \hdots)
   \]
where $\phi \in G^{2}$ and $a_{j} \in {\mathcal O}_{j-1} \setminus {\mathfrak m}$,
$b_{j} \in {\mathcal O}_{j-1} \cap {\mathfrak m}$ for any $3 \leq j \leq n$.
The construction in Proposition \ref{pro:mas2} implies that
$G^{n}$ is a solvable group of derived length $2n+1$ for any $n \geq 3$.
\end{rem}

\bibliography{rendu}
\end{document}